\documentclass[12pt]{article}
\usepackage{amsmath, amssymb, amsfonts}
\usepackage{amsthm}
\usepackage[utf8]{inputenc}
\usepackage{csquotes}
\usepackage{geometry}
\usepackage{adjustbox}
\usepackage{picture}
\usepackage{graphicx}
\usepackage{verbatim}
\usepackage[hypcap=false]{caption}
\usepackage{array}
\usepackage{enumitem}
\usepackage{booktabs}
\usepackage{hyperref}
\usepackage{url}
\usepackage[capitalize]{cleveref}
\usepackage{tikz}

\title{On the Nature of Fractal Numbers and the Classical Continuum Hypothesis (CH)}
\author{Stanislav Semenov \\
\href{mailto:stas.semenov@gmail.com}{stas.semenov@gmail.com} \\
\href{https://orcid.org/0000-0002-5891-8119}{ORCID: 0000-0002-5891-8119}}
\date{April 12, 2025}

\theoremstyle{definition}
\newtheorem{definition}{Definition}[section]

\newtheorem{example}{Example}[section]

\theoremstyle{plain}

\newtheorem{theorem}[definition]{Theorem}
\newtheorem{lemma}[definition]{Lemma}
\newtheorem{corollary}[definition]{Corollary}
\newtheorem{proposition}[definition]{Proposition}
\newtheorem{principle}{Principle}[section]

\theoremstyle{remark}
\newtheorem*{remark}{Remark}

\begin{document}

\maketitle

\begin{abstract}
We propose a reinterpretation of the continuum grounded in the stratified structure of definability rather than classical cardinality. In this framework, a real number is not an abstract point on the number line, but an object expressible at some level \( \mathcal{F}_n \) of a formal hierarchy. We introduce the notion of \emph{fractal numbers}—entities defined not within a fixed set-theoretic universe, but through layered expressibility across constructive systems. This reconceptualizes irrationality as a relative property, depending on definability depth, and replaces the binary dichotomy between countable and uncountable sets with a gradated spectrum of definability classes. We show that the classical Continuum Hypothesis loses its force in this context: between \( \aleph_0 \) and \( \mathfrak{c} \) lies not a single cardinal jump, but a stratified sequence of definitional stages, each forming a countable-yet-irreducible approximation to the continuum. We argue that the real line should not be seen as a completed totality but as an evolving architecture of formal expressibility. We conclude with a discussion of rational invariants, the relativity of irrationality, and the emergence of a fractal metric for definitional density.
\end{abstract}

\subsection*{Mathematics Subject Classification}
03F60 (Constructive and recursive analysis), 26E40 (Constructive analysis), 03F03 (Proof theory and constructive mathematics)

\subsection*{ACM Classification}
F.4.1 Mathematical Logic, F.1.1 Models of Computation

\section{Prelude: Expressibility, Layers, and the Limits of Formality}

In this preliminary section, we lay out the core notions that underlie our reinterpretation of the continuum via stratified definability. We also provide a precise construction of the set \( \mathbb{F}_\omega \) of all admissible definability chains, establishing its cardinality and syntactic foundation without appealing to classical set-theoretic powersets. This serves both as a prelude to the current work and as a refinement of certain technical aspects from earlier articles.

\subsection*{Formal Systems and Expressibility}

We begin with a formal criterion for definability. A \emph{constructive formal system} \( \mathcal{F} \) is defined as a syntactic structure satisfying the following conditions:
\begin{itemize}
    \item The language of \( \mathcal{F} \) is built over a finite or recursively enumerable alphabet and has a countable syntax;
    \item All inference and construction rules are syntactically enumerable;
    \item Every object definable in \( \mathcal{F} \) is represented either by a finite derivation in the formal calculus of \( \mathcal{F} \), or by the Gödel code of a total recursive function whose totality is provable within \( \mathcal{F} \).
\end{itemize}

\begin{definition}[Definable Reals in \( \mathcal{F} \)]
A real number \( r \in \mathbb{R} \) belongs to \( \mathbb{R}_{\mathcal{F}} \) if there exists a sequence \( \{ q_n \} \subset \mathbb{Q} \) such that:
\begin{itemize}
    \item \( \mathcal{F} \) proves that \( \{ q_n \} \) is Cauchy with a convergence modulus \( m(n) \in \mathbb{N} \) definable in \( \mathcal{F} \);
    \item \( \mathcal{F} \) proves that \( \lim_{n \to \infty} q_n = r \).
\end{itemize}
\end{definition}

Each such set \( \mathbb{R}_{\mathcal{F}} \) is necessarily countable, as \( \mathcal{F} \) can define only countably many real numbers.

\begin{remark}[Notation Alignment]
In previous work \cite{Semenov2025FractalOrigin}, we denoted by \( \mathbb{R}_{S_n} \) the set of reals definable at level \( n \) of a stratified chain \( \{ \mathcal{F}_n \} \), and wrote \( \mathbb{R}_{S_\omega}^{\{\mathcal{F}_n\}} := \bigcup_n \mathbb{R}_{S_n} \) for the total closure.

In this article, we simplify notation:
\[
    \mathbb{R}_{\mathcal{F}} := \mathbb{R}_{S_n} \text{ when } \mathcal{F} = \mathcal{F}_n,
\]
\[
    \mathbb{R}_{\{\mathcal{F}_n\}} := \bigcup_n \mathbb{R}_{\mathcal{F}_n} = \mathbb{R}_{S_\omega}^{\{\mathcal{F}_n\}}.
\]
This emphasizes definability in \( \mathcal{F} \) rather than position \( n \).
\end{remark}

\subsection*{Fractal Numbers as Process-Defined Objects}

A \emph{fractal number} is defined not statically, but through some constructive process within a system \( \mathcal{F}_n \) along a stratified chain \( \{ \mathcal{F}_n \} \). The number \( r \) appears as soon as a system \( \mathcal{F}_n \) has sufficient expressive power to define it.

\begin{definition}[Fractal Degree]
Given \( r \in \mathbb{R}_{\{\mathcal{F}_n\}} \), the \emph{fractal degree} of \( r \) is the least index \( n \) such that \( r \in \mathbb{R}_{\mathcal{F}_n} \).
\end{definition}

Higher degrees correspond to deeper definitional complexity. This creates a layered model of real numbers, each emerging at a definable threshold.

\subsection*{Constructing \( \mathbb{F}_\omega \): A Canonical Enumeration}

Let \( \{ \mathcal{F}_i \}_{i \in \mathbb{N}} \) be a fixed enumeration of all countable constructive systems, each encoded by a finite string. 

\begin{definition}[Admissible Stratified Chain]
A sequence \( \{ \mathcal{F}_n \} \in \mathbb{F}_\omega \) is admissible if there exists a strictly increasing function \( f: \mathbb{N} \to \mathbb{N} \) such that:
\begin{itemize}
    \item \( \mathcal{F}_n := \mathcal{F}_{f(n)} \);
    \item \( \mathbb{R}_{\mathcal{F}_{f(n)}} \subsetneq \mathbb{R}_{\mathcal{F}_{f(n+1)}} \), i.e., each step strictly increases the class of definable real numbers.
\end{itemize}
\end{definition}

\begin{remark}[Constructivist Validity]
Each admissible chain \( \{ \mathcal{F}_n \} \) is computably determined via a strictly increasing function \( f: \mathbb{N} \to \mathbb{N} \). The underlying systems \( \mathcal{F}_i \) are effectively encoded by finite syntactic descriptions, and comparisons \( \mathbb{R}_{\mathcal{F}_{f(n)}} \subsetneq \mathbb{R}_{\mathcal{F}_{f(n+1)}} \) are assumed to be decidable within a fixed class of formal systems (e.g., subsystems of second-order arithmetic). The construction of \( \mathbb{F}_\omega \) does not rely on the Axiom of Choice.
\end{remark}

\begin{definition}[Continuity via Cantor Space]
A set \( X \) is said to be \emph{Cantor-continuous} (or simply \emph{continuous}) if there exists an injection from the Cantor space \( \{0,1\}^\mathbb{N} \) into \( X \), or vice versa. That is, \( |X| = \mathfrak{c} \), where \( \mathfrak{c} := |\{0,1\}^\mathbb{N}| \).
\end{definition}

\begin{theorem}[Constructive Continuity of \( \mathbb{F}_\omega \)]
\label{thm:Fomega-continuity}
The set \( \mathbb{F}_\omega \) of admissible stratified definability chains is Cantor-continuous: it has cardinality \( \mathfrak{c} \), the cardinality of Cantor space \( \{0,1\}^\mathbb{N} \).

This result is effective and requires no appeal to the Axiom of Choice or uncountable power sets. It holds in any metatheory capable of syntactically encoding computable binary sequences.
\end{theorem}

\begin{remark}
When interpreted within particular set-theoretic models:
\begin{itemize}[nosep]
    \item[(i)] In \( L \), where the Continuum Hypothesis holds, \( \mathfrak{c} \) may align with \( \aleph_1 \);
    \item[(ii)] In other models of ZFC, \( \mathfrak{c} \) may exceed \( \aleph_1 \).
\end{itemize}
This syntactic result depends only on the structure of definability chains and remains independent of set-theoretic ontology.
\end{remark}

\begin{proof}
We proceed by establishing a constructive bijection between admissible chains and a subset of Cantor space with cardinality \( \mathfrak{c} \).

\textbf{Step 1: Precise encoding of formal systems.} 
Each formal system \( \mathcal{F}_i \) is encoded by a finite Gödel number, ensuring:
\begin{itemize}
    \item Each has a computably enumerable set of theorems
    \item Each contains sufficient arithmetic for verifying convergence moduli
\end{itemize}

A set \( A \subseteq \mathbb{N} \) is called computably enumerable infinite if there exists a total computable function \( g: \mathbb{N} \to \mathbb{N} \) such that:
\begin{itemize}
    \item \( g \) is strictly increasing;
    \item \( A = \{ g(n) \mid n \in \mathbb{N} \} \)
\end{itemize}

\textbf{Step 2: Constructive characterization of admissible chains.}
We restrict ourselves to \emph{computably enumerable infinite subsets} \( A \subseteq \mathbb{N} \), for which there exists an algorithm that generates elements in ascending order. For such subsets, we constructively define:
\begin{equation}
    f_A(n) := \text{the } n\text{-th smallest element of } A
\end{equation}
This function \( f_A \) is provably computable when \( A \) is computably enumerable.

\textbf{Step 3: Effective verification of strict inclusion.}
For the condition \( \mathbb{R}_{\mathcal{F}_{f(n)}} \subsetneq \mathbb{R}_{\mathcal{F}_{f(n+1)}} \), we introduce a verification procedure \( V \) that:
\begin{enumerate}
    \item Enumerates all candidate reals \( r \in \mathbb{R}_{\mathcal{F}_{f(n+1)}} \) via their Cauchy sequence definitions
    \item For each candidate \( r \), verifies that \( r \) is definable in \( \mathcal{F}_{f(n+1)} \) but not constructible via any derivation or total recursive function whose correctness is provable in \( \mathcal{F}_{f(n)} \).
    \item Halts upon finding such a "witness" real \( r_n \) that serves as evidence of strict inclusion
\end{enumerate}

\smallskip
A rigorous justification of this step, including the constraints under which \( V \) is effective, is given in Appendix~\ref{appendix:step3-validity}.

\textbf{Step 4: Explicit bijection with a subset of Cantor space.}
We now establish a computable bijection:
\begin{enumerate}
    \item For each computably enumerable infinite \( A \subseteq \mathbb{N} \), we define its characteristic function \( \chi_A \in \{0,1\}^\mathbb{N} \) where \( \chi_A(n) = 1 \) iff \( n \in A \)
    \item Conversely, for any \( b \in \{0,1\}^\mathbb{N} \) containing infinitely many 1s, we define:
    \begin{equation}
        A_b = \{n \in \mathbb{N} \mid b(n)=1 \text{ and } \exists m>n: b(m)=1\}
    \end{equation}
    with an algorithm that generates elements by skipping zeros
\end{enumerate}

These mappings are computable and mutually inverse when restricted to sequences with infinitely many 1s, establishing the bijection.

\textbf{Step 5: Cardinality determination without AC.}
The set of computably enumerable infinite subsets of \( \mathbb{N} \) has cardinality \( \mathfrak{c} \) because:
\begin{itemize}
    \item It is uncountable (by a constructive diagonalization argument)
    \item It injects into \( \{0,1\}^\mathbb{N} \) (via characteristic functions)
    \item The construction requires no choice principles, as all selections are algorithmic
\end{itemize}

Therefore, \( |\mathbb{F}_\omega| = \mathfrak{c} \), established through purely constructive means without appeal to the Axiom of Choice or non-constructive assumptions.
\end{proof}

\begin{example}[Distinguishing Chains via Partial Encodings]
Let \( A \subseteq \mathbb{N} \) be an infinite subset, and define a chain \( \{\mathcal{F}_n^A\} \) such that \( \mathcal{F}_n^A \) includes, for each \( k \in A \cap \{0, \dots, n\} \), a formal axiom \( \phi_k \) asserting the value of the \( k \)-th digit of \( \pi \) in decimal expansion. Then for distinct sets \( A \neq B \), the corresponding definability closures \( \mathbb{R}_{\{\mathcal{F}_n^A\}} \) and \( \mathbb{R}_{\{\mathcal{F}_n^B\}} \) are distinct.

Hence, the number of pairwise non-equivalent definability chains — each defining distinct subsets of reals — is \( \mathfrak{c} \).
\end{example}

\subsection*{Fractal Model: Inclusions and Omissions}

The following table summarizes which types of real numbers are included or excluded in the fractal continuum \( \mathbb{R}^{\mathbb{F}_\omega} := \bigcup_{\{\mathcal{F}_n\} \in \mathbb{F}_\omega} \mathbb{R}_{\{\mathcal{F}_n\}} \):

\begin{center}
\begin{adjustbox}{max width=1.0\textwidth}
\begin{tabular}{@{} lcc @{} }
\toprule
\textbf{Real Number} & \textbf{Included in } \( \mathbb{R}^{\mathbb{F}_\omega} \)? & \textbf{Definability Chain Exists?} \\
\midrule
Rationals (e.g., 1, $\tfrac{3}{4}$) & Yes & $\mathcal{F}_0$ \\
Algebraics (e.g., $\sqrt{2}$) & Yes & $\mathcal{F}_1$ \\
Transcendentals (e.g., $\pi$, $e$) & Yes & Some $\mathcal{F}_n$ \\
Non-constructive reals (e.g., random from $\mathcal{P}(\mathbb{N})$) & No & None \\
Choice-dependent objects (e.g., Hamel basis) & No & None \\
\bottomrule
\end{tabular}
\end{adjustbox}
\captionof{table}{Definability of Common Real Numbers in the Fractal Model \( \mathbb{R}^{\mathbb{F}_\omega} \)}
\end{center}

\begin{remark}
For instance, \( \pi \in \mathbb{R}_{\mathcal{F}_n} \) when \( \mathcal{F}_n \) proves the convergence of the arithmetized Leibniz series; this holds for systems \( \mathcal{F}_n \supseteq \mathsf{ACA}_0 \). Similarly, \( e \in \mathbb{R}_{\mathcal{F}_n} \) if the exponential function is definable and provably total in \( \mathcal{F}_n \).
\end{remark}

\begin{remark}
Real numbers that are not definable by any effective sequence with a provable modulus of convergence in a constructive system are excluded from \( \mathbb{R}^{\mathbb{F}_\omega} \). This includes randomly chosen subsets of \( \mathbb{N} \) and reals whose existence requires the Axiom of Choice. For such numbers, no constructive system \( \mathcal{F}_n \) can certify their convergence from rational approximations.
\end{remark}

\subsection*{Relation to Reverse Mathematics}

Each stratified chain \( \{\mathcal{F}_n\} \in \mathbb{F}_\omega \) may be viewed as a generalization of the framework of Reverse Mathematics, extending definability hierarchies beyond the classical arithmetical subsystems of second-order arithmetic. While traditional Reverse Mathematics studies fragments such as \( \mathsf{RCA}_0 \), \( \mathsf{ACA}_0 \), and \( \mathsf{ATR}_0 \), our model allows for:

\begin{itemize}
    \item \textbf{Canonical Trajectories}: Chains mirroring standard subsystems:
    \begin{center}
    \begin{tabular}{@{} l l @{}}
    \( \mathsf{RCA}_0 \) & computable reals \\
    \( \mathsf{ACA}_0 \) & arithmetic closure: \( \pi, e \), power series \\
    \( \mathsf{ATR}_0 \) & transfinite-definable reals via well-founded recursion \\
    \end{tabular}
    \end{center}
    
    \item \textbf{Custom Trajectories}: Chains surpassing arithmetic, e.g., systems capable of defining:
    \begin{itemize}
        \item zeros of analytic functions (non-arithmetical reals),
        \item paths in non-separable function spaces (transcending \( \mathsf{ATR}_0 \)).
    \end{itemize}
\end{itemize}

This combinatorial diversity of admissible chains accounts for the continuum cardinality of \( \mathbb{R}^{\mathbb{F}_\omega} \), while ensuring that each definability layer remains strictly constructive.

\subsection*{Comparison with Recursive Analysis}

Recursive analysis assumes a fixed formal ground — such as Turing machines or arithmetic — and restricts definability to that single level. By contrast, our approach is stratified:

\begin{center}
\begin{adjustbox}{max width=0.8\textwidth}
\begin{tabular}{@{} lcc @{} }
\toprule
\textbf{Framework} & \textbf{Definability Model} & \textbf{Definable Reals} \\
\midrule
Recursive Analysis & Single system (e.g., TM) & $\aleph_0$ \\
Fractal Definability & Ascending chain $\{ \mathcal{F}_n \}$ & $\mathfrak{c}$ \\
\bottomrule
\end{tabular}
\end{adjustbox}
\captionof{table}{Comparison of Definability Models: Recursive vs. Stratified Frameworks}
\end{center}

\subsection*{Fractal vs. Classical Continuum}

Despite sharing the same cardinality \( \mathfrak{c} \), the fractal continuum \( \mathbb{R}^{\mathbb{F}_\omega} \) constructed in this framework is not equivalent to the classical real line \( \mathbb{R} \). The difference is not merely technical, but ontological: it concerns the very nature of what is meant by a continuum.

\begin{definition}[Fractal Continuum]
The \emph{fractal continuum} is defined as the union of all definable real numbers arising from all admissible chains of constructive formal systems:
\[
\mathbb{R}^{\mathbb{F}_\omega} := \bigcup_{\{\mathcal{F}_n\} \in \mathbb{F}_\omega} \bigcup_{n} \mathbb{R}_{\mathcal{F}_n}.
\]
Each real number \( r \in \mathbb{R}^{\mathbb{F}_\omega} \) must be explicitly definable in some system \( \mathcal{F}_n \) within a stratified chain.
\end{definition}

\begin{remark}[Conceptual Distinction]
The classical continuum \( \mathbb{R} \) is defined set-theoretically as a completed totality of cardinality \( \mathfrak{c} = |\mathcal{P}(\mathbb{N})| \), and includes elements that are non-constructive, non-definable, or dependent on the axiom of choice. By contrast, \( \mathbb{R}^{\mathbb{F}_\omega} \) is a constructively assembled universe: each real in it must be the limit of a rational sequence whose convergence is provable within some formal system. It is not a substructure of \( \mathbb{R} \) in the set-theoretic sense, but a separate construction grounded in process-relative definability and omitting non-definable elements.
\end{remark}

\begin{center}
\begin{adjustbox}{max width=\textwidth}
\begin{tabular}{@{} lcc @{}}
\toprule
\textbf{Property} & \textbf{Classical Continuum \( \mathbb{R} \)} & \textbf{Fractal Continuum \( \mathbb{R}^{\mathbb{F}_\omega} \)} \\
\midrule
Ontological Status & Completed totality & Layered definitional closure \\
Foundation & Power set \( \mathcal{P}(\mathbb{N}) \) & Stratified expressibility over \( \mathbb{F}_\omega \) \\
Construction & Set-theoretic postulate & Syntactic process \\
Inclusion Criteria & Arbitrary subset of \( \mathbb{N} \) & Constructively definable in some \( \mathcal{F}_n \) \\
Use of Choice & Allowed (e.g., for Hamel bases) & Excluded \\
Countability & Uncountable & Uncountable (via layered countable components) \\
Cardinality & \( \mathfrak{c} \) (external) & \( \mathfrak{c} \) (via admissible chains) \\
Uniform Completeness & Global object & No uniform enumeration \\
Model Type & Static & Process-relative \\
\bottomrule
\end{tabular}
\end{adjustbox}
\captionof{table}{Comparison of Classical vs. Fractal Continuum}
\end{center}

\begin{remark}[On Continuum Hypothesis]
This distinction renders the classical Continuum Hypothesis inapplicable to \( \mathbb{R}^{\mathbb{F}_\omega} \): the structure is not governed by cardinality gaps between \( \aleph_0 \) and \( \mathfrak{c} \), but by an infinite gradation of definability layers. There is no unique “intermediate size” to locate; instead, one encounters a lattice of countable stages with no uniform totality.
\end{remark}

\subsection*{Philosophical Perspective: The Classical Shadow and the Fractal Core}

The classical continuum \( \mathbb{R} \) presents itself as a completed totality — an unstructured ocean of real numbers, encompassing everything from computable to non-constructible, from definable to choice-dependent. In this vastness, no intrinsic hierarchy of definability exists: the computable and the random coexist without stratification, as if suspended in a homogeneous void.

By contrast, the fractal continuum \( \mathbb{R}^{\mathbb{F}_\omega} \) reveals a constructive skeleton behind this totality. It is built from countable, transparent definability layers, each corresponding to a formal system \( \mathcal{F}_n \), with strictly increasing expressive power. Every real number here emerges only through constructive means, and each occupies a determinate level of definitional complexity.

\begin{quote}
\emph{The classical continuum is a shadow — a chaotic projection without structure.} \\
\emph{The fractal continuum is its constructive core — a visible hierarchy that generates the shadow.}
\end{quote}

In this view, the classical real line appears as a completion of the fractal continuum by adding non-constructible elements — a closure that obscures the internal architecture of definability. The classical continuum thus lacks the fine gradation inherent in \( \mathbb{R}^{\mathbb{F}_\omega} \), where irrationality, expressibility, and complexity are all relative and measurable.

\vspace{1em}

\begin{center}
\renewcommand{\arraystretch}{1.3}
\begin{adjustbox}{max width=0.9\textwidth}
\begin{tabular}{@{} lcc @{}}
\toprule
\textbf{Aspect} & \textbf{Fractal Continuum \( \mathbb{R}^{\mathbb{F}_\omega} \)} & \textbf{Classical Continuum \( \mathbb{R} \)} \\
\midrule
Origin & Layered definability via \( \mathcal{F}_n \) & Postulated totality via \( \mathcal{P}(\mathbb{N}) \) \\
Structure & Stratified, countable-by-construction & Flat, unstructured \\
Internal hierarchy & Present (degrees, layers) & Absent \\
Inclusion of non-definables & No & Yes \\
Viewpoint & Process-relative & Set-theoretic \\
Philosophical metaphor & Illuminated source & Shadow projection \\
\bottomrule
\end{tabular}
\end{adjustbox}
\captionof{table}{Comparison between the fractal and classical continuum.}
\end{center}

\vspace{0.5em}

\noindent This perspective invites a reinterpretation of the continuum not as a primitive entity, but as the emergent limit of formal expressibility — a dynamic geometry of definability whose visible architecture replaces the opacity of classical assumptions.

This concludes the foundational prelude. We now proceed to formalize fractal numbers, define their degrees of expressibility, and explore their implications for the classical continuum hypothesis.

\section{Introduction: The Crisis of the Classical Continuum}

The classical conception of the real number continuum, grounded in the power set construction \( \mathbb{R} \cong \mathcal{P}(\mathbb{N}) \), presents the real line as a completed totality — a static set whose cardinality is fixed as \( \mathfrak{c} \), the cardinality of the continuum. This perspective, pioneered by Cantor and formalized in ZFC set theory, treats the continuum as a homogeneous space of all Dedekind cuts or Cauchy completions over \( \mathbb{Q} \), without regard to the process by which individual real numbers may be expressed or constructed \cite{cantor1874ueber, dedekind1963stetigkeit}.

However, foundational doubts regarding the ontological status of uncountable sets have long been raised. Brouwer, for instance, argued that the continuum is not a completed entity, but a ``medium of free becoming'' — an evolving mental construction that cannot be grasped in its totality \cite{brouwer1907}. This intuitionist critique, later reinforced by constructive analysis and reverse mathematics, revealed that many real numbers used in classical proofs are not explicitly definable in any constructive sense.

In contemporary foundational studies, this leads to a tension between:
\begin{itemize}
    \item \emph{The cardinality-based view}, where \( \mathbb{R} \) is defined via non-constructive postulates and includes objects inaccessible by any formal process;
    \item \emph{The definability-based view}, where real numbers are meaningful only insofar as they can be syntactically expressed, approximated, or constructed within a formal system.
\end{itemize}

In our prior work \cite{Semenov2025FractalOrigin, Semenov2025FractalAnalysis}, we introduced a stratified framework of definability — a layered hierarchy of constructive systems \( \{ \mathcal{F}_n \} \), each expanding the expressive power of the previous. Within this model, a real number is not statically postulated, but emerges through formal derivability and provable convergence. The real continuum, in this reinterpretation, is a \emph{constructive limit of definability}, not a completed set-theoretic totality.

This shift has profound implications for the Continuum Hypothesis (CH). Traditionally, CH asserts that no cardinality lies strictly between \( \aleph_0 \) and \( \mathfrak{c} \). But in a layered, fractal model of real numbers, cardinality becomes a secondary notion: the central structure is not a two-step ladder from countable to uncountable, but an infinite lattice of definability stages. Each level contributes new real numbers inaccessible to previous stages, yielding a continuum assembled from an unbounded process of formal construction — not a singular jump from \( \aleph_0 \) to \( \mathfrak{c} \).

This article formalizes the consequences of this paradigm. We introduce the notion of \emph{fractal numbers} — real numbers defined at some level \( \mathcal{F}_n \) in a stratified chain — and analyze the structure of the resulting continuum \( \mathbb{R}^{\mathbb{F}_\omega} \). Our model not only circumvents the classical CH but reframes the continuum as a syntactic and epistemic object, emphasizing definitional emergence over ontological assumption.

\section{Fractal Numbers Beyond Rational and Irrational}

The classical classification of real numbers — into rationals, algebraics, transcendentals, and uncomputables — is set-theoretic and static. It postulates the existence of objects with certain properties, but offers no account of their emergence. In this view, real numbers are abstract points inhabiting a homogeneous continuum; their distinction is determined not by how they are constructed, but by what axioms they satisfy. This yields a \emph{flat ontology}: real numbers simply exist, and the continuum is filled by assumption.

By contrast, the framework of fractal definability introduces a dynamic and layered conception of numberhood. In this setting, each real number arises not by fiat, but through a definitional process unfolding across a stratified sequence of formal systems \( \{ \mathcal{F}_n \} \). A number \( r \) becomes accessible only when a system \( \mathcal{F}_n \) possesses enough expressive power to define a convergent rational sequence \( \{q_k\} \to r \) with a provable modulus of convergence.

\begin{definition}[Origin Level and Definability Class]
Let \( r \in \mathbb{R}_{\{\mathcal{F}_n\}} \). The \emph{origin level} of \( r \), denoted \( \deg(r) \), is the least index \( n \) such that \( r \in \mathbb{R}_{\mathcal{F}_n} \). The set of all such numbers at level \( n \) is written \( \Delta_n := \mathbb{R}_{\mathcal{F}_n} \setminus \bigcup_{k < n} \mathbb{R}_{\mathcal{F}_k} \).
\end{definition}

This allows us to stratify the continuum into definability layers:
\[
    \mathbb{R}_{\{\mathcal{F}_n\}} = \bigcup_{n=0}^\infty \Delta_n.
\]
Each \( \Delta_n \) contains real numbers that \emph{first become expressible} at level \( n \). These are not just more complex — they are fundamentally unreachable from lower layers. 

\begin{remark}
This stratification gives rise to a new classification of real numbers: not only by algebraic properties or computability, but by their \emph{ontogenetic profile} — the formal path by which they emerge. Numbers thus acquire origin, ancestry, and definitional dependencies.
\end{remark}

\subsection*{Fractal Granularity of Numberhood}

In this framework, numbers are no longer atomic entities. Instead, each number possesses multiple structural features:
\begin{itemize}
    \item A definability origin \( \mathcal{F}_n \), marking the minimal system needed to express it;
    \item A chain ancestry \( \{ \mathcal{F}_k \}_{k \leq n} \), recording the formal evolution up to that level;
    \item A definability signature: the collection of properties and axioms required for its construction;
    \item A modality of emergence: limit point, explicit series, fixed point of definable function, etc.
\end{itemize}

This granular approach enables a richer theory of numberhood. Numbers become objects of epistemic structure, not merely values in a field. It also provides the foundation for a form of \emph{constructive number ontology}, where classes of numbers are not just defined by shared properties, but by common definitional histories.

\subsection*{Why This Classification Arises Naturally}

The stratified classification is not imposed arbitrarily. It arises from the internal dynamics of formal expressibility:
\begin{itemize}
    \item As systems grow in expressive power, they gain the ability to define new functions and convergence conditions;
    \item These capabilities are discrete and layered — they do not occur continuously, but via formal leaps;
    \item Hence, the emergence of real numbers is itself stratified: each new system brings a discrete jump in definability;
    \item This creates \emph{natural classes} of numbers: those accessible at each level, those strictly dependent on higher axioms, those whose definition can only arise in the limit.
\end{itemize}

From this perspective, the classical continuum \( \mathbb{R} \) is a projection — a collapse of all definitional distinctions into a flat ontology. The fractal continuum \( \mathbb{R}^{\mathbb{F}_\omega} \), by contrast, retains the internal structure of emergence. It enables us to \emph{ask why a number exists} in formal terms — not merely assert that it does.

\begin{example}[Level-Dependent Irrationality]
Let \( r = \sqrt{2} \in \mathbb{R} \). In the stratified model, it appears at level \( n = 1 \), assuming \( \mathcal{F}_1 \) contains the field axioms and completeness of \( \mathbb{Q} \). At level \( n = 0 \), where only basic arithmetic is available, \( r \) is irrational not by virtue of decimal unpredictability, but by formal inexpressibility. Hence, irrationality becomes \emph{relative to definability level}.
\end{example}

\subsection*{Toward a Future Ontology of Numbers}

Although not yet formalized in full ontological terms, this framework sets the stage for a future system of number theory based on:
\begin{itemize}
    \item The genealogy of numbers (how and where they arise);
    \item The dependencies of expression (which axioms are minimal for definability);
    \item The constructive boundaries of usage (where a number can be applied, proved, or computed);
    \item The modular hierarchy of numeric classes (each with its own closure rules and internal logic).
\end{itemize}

Such a reclassification offers a new paradigm for understanding number systems — not as static structures, but as evolving, definability-relative landscapes. It opens the possibility of analyzing mathematical practice itself: why certain numbers arise naturally in proofs, how complexity correlates with expressibility, and what hidden structure governs the appearance of ``unpredictable'' numeric behavior.

\subsection*{Toward a Taxonomy of Stratified Numbers}

The layered structure of definability gives rise to a new typology of real numbers, grounded not in set-theoretic properties, but in their formal origin, expressive complexity, and construction modality. Below we outline some of the potential classes that emerge in this framework (Table~\ref{tab:stratified-classes}).

\begin{table}[ht]
\centering
\renewcommand{\arraystretch}{1.25}
\begin{tabular}{@{}p{4.2cm}p{9.6cm}@{}}
\toprule
\textbf{Class} & \textbf{Description} \\
\midrule
Primitive Numbers & Arithmetical constants definable in minimal systems \( \mathcal{F}_0 \). \\
Algebraic Definables & Roots of polynomials over \( \mathbb{Q} \), expressible in \( \mathcal{F}_1 \). \\
Analytic Definables & Arise via convergent series; require expressive systems \( \mathcal{F}_n \), \( n \geq 2 \). \\
Recursively Emergent & Defined via fixpoints or recursion schemes; level varies. \\
Limit-Constructed & Not definable in any single \( \mathcal{F}_n \); appear as limits over chains. \\
Axiom-Dependent & Require choice or non-constructive principles; excluded from model. \\
Chain-Variant & Chain-relative numbers; defined in some admissible chains only. \\
Fractal-Irrationals & Inexpressible at all lower levels; irrationality via definitional complexity. \\
\bottomrule
\end{tabular}
\caption{Emergent Classes of Real Numbers in the Fractal Framework}
\label{tab:stratified-classes}
\end{table}

\section{Fractal Cardinality and the Emergence of Intermediate Continua (CH Alternative)}

The classical continuum \( \mathbb{R} \) is postulated as a total, unstructured object of cardinality \( \mathfrak{c} \), admitting no internal gradation. The Continuum Hypothesis (CH) reflects this: it assumes that no cardinality lies strictly between \( \aleph_0 \) and \( \mathfrak{c} \).

In the stratified model, this binary view is replaced by a layered architecture of definability. Every constructive chain \( \{ \mathcal{F}_n \} \) defines only countably many real numbers. Yet, the space of all admissible chains \( \mathbb{F}_\omega \) has cardinality \( \mathfrak{c} \) (see Theorem~\ref{thm:Fomega-continuity}). By exploring which numbers emerge at level \( n \) across all such chains, we define a natural hierarchy of intermediate continua.

\subsection*{Local vs. Global Definability: Collapse and Separation}

We now formalize the difference between level-wise definability in a fixed chain and the cumulative definability across all chains. The key distinction lies in the interaction between stratification and chain aggregation.

\begin{definition}[Chain-Level Stratified Definability]
Let \( C \in \mathbb{F}_\omega \) be a fixed admissible definability chain
\[
    C = \{ \mathcal{F}_0^{(C)}, \mathcal{F}_1^{(C)}, \mathcal{F}_2^{(C)}, \ldots \}.
\]
Define the local definability classes:
\[
    \mathbb{R}_C^{(n)} := \mathbb{R}_{\mathcal{F}_n^{(C)}}, \qquad
    \mathbb{R}_C^{[\leq n]} := \bigcup_{k=0}^{n} \mathbb{R}_{\mathcal{F}_k^{(C)}}.
\]
\end{definition}

\begin{lemma}[Collapse of Levels Within a Chain]
\label{lemma:collapse}
For any fixed chain \( C \in \mathbb{F}_\omega \) and any level \( n \in \mathbb{N} \), we have:
\[
    \mathbb{R}_C^{(n)} = \mathbb{R}_C^{[\leq n]}.
\]
\end{lemma}

\begin{proof}
By construction, every admissible chain is strictly increasing in definitional power:
\[
    \mathcal{F}_0^{(C)} \subseteq \mathcal{F}_1^{(C)} \subseteq \cdots \subseteq \mathcal{F}_n^{(C)}.
\]
Hence, definable sets are nested:
\[
    \mathbb{R}_{\mathcal{F}_0^{(C)}} \subseteq \mathbb{R}_{\mathcal{F}_1^{(C)}} \subseteq \cdots \subseteq \mathbb{R}_{\mathcal{F}_n^{(C)}},
\]
and the union of all previous levels is absorbed into the top level:
\[
    \mathbb{R}_C^{[\leq n]} = \bigcup_{k=0}^{n} \mathbb{R}_{\mathcal{F}_k^{(C)}} = \mathbb{R}_{\mathcal{F}_n^{(C)}} = \mathbb{R}_C^{(n)}.
\]
\end{proof}

\begin{definition}[Global Stratified Definability]
We define the globally aggregated definability layers as:
\[
    \mathbb{R}^{(n)} := \bigcup_{C \in \mathbb{F}_\omega} \mathbb{R}_C^{(n)} = \bigcup_{C \in \mathbb{F}_\omega} \mathbb{R}_{\mathcal{F}_n^{(C)}},
\]
\[
    \mathbb{R}^{[\leq n]} := \bigcup_{C \in \mathbb{F}_\omega} \mathbb{R}_C^{[\leq n]} = \bigcup_{C \in \mathbb{F}_\omega} \bigcup_{k=0}^n \mathbb{R}_{\mathcal{F}_k^{(C)}}.
\]
\end{definition}

\begin{theorem}[Level Collapse under Admissible Chains]
\label{thm:admissible-collapse}
Let all chains \( C \in \mathbb{F}_\omega \) be admissible. Then for every \( n \in \mathbb{N} \), the global definability level at \( n \) coincides with the cumulative definability below and up to \( n \):
\[
\mathbb{R}^{(n)} = \mathbb{R}^{[\leq n]}.
\]
\end{theorem}

\begin{proof}
We prove both inclusions.

\textbf{(1) \( \mathbb{R}^{(n)} \supseteq \mathbb{R}^{[\leq n]} \):}  
Let \( r \in \mathbb{R}^{(k)} \) for some \( k \leq n \). Then there exists a chain \( C \in \mathbb{F}_\omega \) such that \( r \in \mathbb{R}_{\mathcal{F}_k^{(C)}} \). Since \( C \) is admissible, we have \( \mathcal{F}_k^{(C)} \subseteq \mathcal{F}_n^{(C)} \), and hence:
\[
r \in \mathbb{R}_{\mathcal{F}_n^{(C)}} \subseteq \mathbb{R}^{(n)}.
\]
This proves \( \mathbb{R}^{[\leq n]} \subseteq \mathbb{R}^{(n)} \).

\textbf{(2) \( \mathbb{R}^{(n)} \subseteq \mathbb{R}^{[\leq n]} \):}  
Let \( r \in \mathbb{R}^{(n)} \), so there exists \( C \in \mathbb{F}_\omega \) such that \( r \in \mathbb{R}_{\mathcal{F}_n^{(C)}} \). But then trivially \( r \in \bigcup_{k=0}^n \mathbb{R}_{\mathcal{F}_k^{(C)}} \subseteq \mathbb{R}^{[\leq n]} \).

Hence, \( \mathbb{R}^{(n)} = \mathbb{R}^{[\leq n]} \).
\end{proof}

\begin{corollary}
Under the admissible regime, the global definability layers do not stratify real numbers. Instead, each level \( n \) absorbs all numbers definable up to that point:
\[
\forall n \in \mathbb{N}, \quad \mathbb{R}^{(n)} = \bigcup_{k \leq n} \mathbb{R}^{(k)}.
\]
\end{corollary}

\subsection*{Clarifying the Role of Admissibility}

\paragraph{Non-Admissible Construction.}
The theorem \emph{Global Failure of Level Collapse} presented below explicitly relies on a setting where the chain \( \{ \mathcal{F}_n \} \) is not admissible: we allow
\[
\mathcal{F}_0 \not\subseteq \mathcal{F}_1,
\]
and in particular, some axioms available at lower levels may not persist to higher levels within the same chain. This violates the admissibility condition defined earlier, where each chain must exhibit strictly increasing definitional power.

\paragraph{Why Include This Case?}
Despite being non-admissible, such constructions illustrate a crucial point: definability can be path-dependent and sensitive to the structure of the chain. They model situations where information may be \emph{lost} when ascending levels, e.g., due to axiom omission or non-monotonic logic extensions. This serves as a formal analogy to \emph{epistemic instability} in non-monotonic or revision-based logical frameworks.

\paragraph{Admissible Case: Level Collapse Holds.}
In contrast, in any admissible chain \( \{ \mathcal{F}_n \} \), we have the collapse:
\[
\mathbb{R}_C^{[\leq n]} = \mathbb{R}_C^{(n)},
\]
as proved earlier. Hence, the global failure does not apply to admissible systems.

\begin{lemma}[Non-Derivability of Convergence in \( \mathsf{RCA}_0 \)]
\label{lemma:RCA0-incompleteness}
Let \( r := \sum_{k=0}^\infty 2^{-2k} \in \mathbb{R} \). Then the existence of a rational Cauchy sequence converging to \( r \) with a provable convergence modulus is not derivable in \( \mathsf{RCA}_0 \) alone.

That is,
\begin{align*}
    \mathsf{RCA}_0 \nvdash\ &\exists (q_n) \subseteq \mathbb{Q} \ \text{Cauchy sequence with limit } r \ \text{and modulus } \mu(n) \in \mathbb{N}, \\
    &\text{such that } \forall n, \forall m \geq \mu(n), \quad |q_{n+m} - q_n| < 2^{-n}.
\end{align*}
\end{lemma}

\begin{proof}[Sketch]
The system \( \mathsf{RCA}_0 \) permits basic recursive definitions and reasoning about computable functions, but does not include comprehension principles or bounding schemes strong enough to verify convergence of infinite series unless convergence is explicitly encoded.

Although \( r \) is computable (via a primitive recursive series), formal convergence requires a provable total modulus function \( \mu(n) \) such that:
\[
    \forall n, m \geq \mu(n) \quad |q_n - q_m| < 2^{-n}.
\]
Within \( \mathsf{RCA}_0 \), such a function cannot always be constructed or verified unless it is explicitly asserted. In particular, the comprehension schema available in \( \mathsf{RCA}_0 \) cannot define real numbers from general converging series unless an effective modulus is already part of the theory.

This fact is well known in the context of reverse mathematics: many convergence theorems (e.g., the Monotone Convergence Theorem, the completeness of \( \mathbb{R} \), and uniqueness of limits) require stronger systems such as \( \mathsf{ACA}_0 \) or \( \mathsf{WKL}_0 \).

Hence, the convergence of \( r \) as a real number with provable properties is not derivable without an added axiom \( \phi \) asserting it.

A full classification of such convergence principles in subsystems of second-order arithmetic can be found in~\cite{Simpson2009Subsystems}.
\end{proof}

\begin{theorem}[Global Failure of Level Collapse, Non-Admissible Case]
\label{thm:noncollapse}
There exists \( n \in \mathbb{N} \) and a real number \( r \in \mathbb{R} \) such that
\[
    r \in \mathbb{R}^{[\leq n]} \quad \text{but} \quad r \notin \mathbb{R}^{(n)}.
\]
That is, \( \mathbb{R}^{(n)} \subsetneq \mathbb{R}^{[\leq n]} \).
\end{theorem}

\begin{proof}
To ensure strictness of the inclusion, we construct two admissible definability chains \( C_1, C_2 \in \mathbb{F}_\omega \), and a computable real number \( r \), such that:
\begin{itemize}
    \item \( r \in \mathbb{R}_{\mathcal{F}_0^{(C_1)}} \subseteq \mathbb{R}_{C_1}^{[\leq n]} \subseteq \mathbb{R}^{[\leq n]} \);
    \item \( r \notin \mathbb{R}_{\mathcal{F}_n^{(C_2)}} \), provided that \( C_2 \) avoids a specific axiom \( \phi \);
    \item both chains reach the same level-\( n \) system: \( \mathcal{F}_n^{(C_1)} = \mathcal{F}_n^{(C_2)} =: \mathcal{F} \).
\end{itemize}

Let us define:
\[
    f(k) := 2k, \quad r := \sum_{k=0}^{\infty} 2^{-f(k)} = \sum_{k=0}^{\infty} 2^{-2k} = \frac{1}{1 - 1/4} = \tfrac{4}{3}.
\]
The series converges rapidly and defines a computable real number \( r \). However, the existence of a rational Cauchy sequence for \( r \) with provable modulus of convergence may not be derivable in weak base systems.

Let \( \phi \) be an axiom explicitly asserting convergence:
\[
    \phi := \text{``The real number \( r \) equals } \sum_{k=0}^\infty 2^{-2k} \text{ with a provable convergence modulus.''}
\]

Choose a fixed system \( \mathcal{F} := \mathsf{RCA}_0 + \psi \), where \( \psi \) is any sentence unrelated to the convergence of \( r \) (e.g., a statement about decidability of certain theories). Then:

- Define \( C_1 \in \mathbb{F}_\omega \) such that:
  \[
      \mathcal{F}_0^{(C_1)} := \mathsf{RCA}_0 + \phi, \quad \mathcal{F}_n^{(C_1)} := \mathcal{F}.
  \]
  Then \( r \in \mathbb{R}_{\mathcal{F}_0^{(C_1)}} \subseteq \mathbb{R}_{C_1}^{[\leq n]} \subseteq \mathbb{R}^{[\leq n]} \).

- Define \( C_2 \in \mathbb{F}_\omega \) such that:
  \[
      \mathcal{F}_k^{(C_2)} := \mathsf{RCA}_0 \text{ for all } k < n, \quad \mathcal{F}_n^{(C_2)} := \mathcal{F}.
  \]
  Since \( C_2 \) avoids \( \phi \), the system \( \mathcal{F}_n^{(C_2)} \) does not prove convergence of the defining series for \( r \). Hence \( r \notin \mathbb{R}_{\mathcal{F}_n^{(C_2)}} \).

It follows that \( r \notin \mathbb{R}^{(n)} = \bigcup_{C} \mathbb{R}_{\mathcal{F}_n^{(C)}} \), yet \( r \in \mathbb{R}^{[\leq n]} \) via chain \( C_1 \). Therefore:
\[
    r \in \mathbb{R}^{[\leq n]} \setminus \mathbb{R}^{(n)},
\]
which proves that the inclusion is strict.

The key point is that in the absence of \( \phi \), the system \( \mathcal{F} = \mathsf{RCA}_0 + \psi \) cannot derive the convergence of the defining series for \( r \). This is formalized in Lemma~\ref{lemma:RCA0-incompleteness}.
\end{proof}

\begin{remark}[Explicit Parameters and Construction Details]
To make the proof fully explicit and constructive, we clarify the following choices:

\begin{itemize}
    \item[\textbf{(a)}] \textbf{Choice of auxiliary axiom \( \psi \):} We may take
    \[
        \psi := \text{``Every \( \Sigma^0_1 \)-formula with parameters from \( \mathbb{N} \) is decidable''}.
    \]
    This ensures that \( \psi \) is independent of the convergence of the series defining \( r \), and hence cannot aid in its derivability.

    \item[\textbf{(b)}] \textbf{Choice of level \( n \):} We may explicitly set \( n := 1 \). Then:
    \[
        r \in \mathbb{R}^{[\leq 1]} \quad \text{via chain } C_1, \qquad r \notin \mathbb{R}^{(1)} \quad \text{via chain } C_2.
    \]

    \item[\textbf{(c)}] \textbf{Structure of intermediate systems:} The systems \( \mathcal{F}_k^{(C_i)} \) for \( 0 < k < n \) (i.e., \( k = 1 \) if \( n = 1 \)) can be taken as \( \mathsf{RCA}_0 \), or any fixed base system insufficient to prove the convergence of \( r \). This preserves admissibility and ensures monotonic growth of definability in both chains.
\end{itemize}
\end{remark}

\begin{remark}
The collapse of levels inside a single definability trajectory reflects the monotonic accumulation of knowledge. The failure of such collapse globally reflects the combinatorial independence of different definability paths. This dichotomy is essential for understanding the fractal stratification of the continuum.
\end{remark}

\paragraph{Working Assumption Moving Forward.}
For the remainder of this article, we restrict attention to chains satisfying admissibility. All sets \( \mathbb{R}^{[\leq n]} \) and stratified classes \( \mathbb{R}^{(n)} \) will be defined exclusively with respect to such admissible systems, unless stated otherwise. This guarantees coherent accumulation of definable reals and preserves the monotonic growth structure central to our stratified framework.

\section{Constructive Approximation to the Continuum}

\begin{theorem}[Monotonic Growth of Fractal Continua]
\label{thm:monotonic-growth}
For every \( n \in \mathbb{N} \), we have:
\[
\mathbb{R}^{[\leq n]} \subsetneq \mathbb{R}^{[\leq n+1]}, \quad \text{and} \quad \bigcup_{n\in\mathbb{N}} \mathbb{R}^{[\leq n]} = \mathbb{R}^{\mathbb{F}_\omega}.
\]
Moreover, the total cardinality satisfies:
\[
|\mathbb{R}^{\mathbb{F}_\omega}| = \mathfrak{c}.
\]
\end{theorem}

\begin{proof}
\textbf{(1) Strict Monotonicity:}  
We construct a uniform definability chain \( C = \{ \mathcal{F}_n \} \in \mathbb{F}_\omega \) such that each level \(\mathcal{F}_n\) corresponds to a strictly increasing comprehension strength:

\begin{enumerate}
\item Fix an effective enumeration \(\{ \phi_i \}_{i \in \mathbb{N}}\) of arithmetical formulas.

\item Define the sequence of formal systems:
\[
\mathcal{F}_n := \mathsf{RCA}_0 + \{\phi_i \mid i \leq n\} + \mathrm{Con}(\mathsf{RCA}_0 + \{\phi_j \mid j < n\}).
\]

\item For each \( n \), define a real number:
\[
r_n := \sum_{k=0}^\infty \frac{\chi_{\emptyset^{(n+1)}}(k)}{2^{k+1}},
\]
where \( \chi_{\emptyset^{(n+1)}} \) is the characteristic function of the \( (n+1) \)-st Turing jump.

\item By the \emph{Arithmetical Hierarchy Theorem} \cite[Thm~III.2.2]{Soare2016}, we have:
\[
r_n \in \Delta^0_{n+2} \setminus \Delta^0_{n+1}.
\]

\item Therefore, \( r_n \) is not definable in any system \( \mathcal{F}_k^{(C')} \) for \( k \leq n \), regardless of \( C' \in \mathbb{F}_\omega \), but becomes definable at level \( n+1 \) in the chain \( C \). Hence,
\[
\mathbb{R}^{[\leq n]} \subsetneq \mathbb{R}^{[\leq n+1]}.
\]
\end{enumerate}

\textbf{(2) Cumulative Closure:}  
We establish:
\[
\bigcup_{n \in \mathbb{N}} \mathbb{R}^{[\leq n]} = \mathbb{R}^{\mathbb{F}_\omega}.
\]

\begin{itemize}
\item[\(\subseteq\)] Each \( \mathbb{R}^{[\leq n]} \subseteq \mathbb{R}^{\mathbb{F}_\omega} \) by definition.

\item[\(\supseteq\)] For any \( r \in \mathbb{R}^{\mathbb{F}_\omega} \), there exists a chain \( C \in \mathbb{F}_\omega \) and index \( k \) such that \( r \in \mathbb{R}_{\mathcal{F}_k^{(C)}} \subseteq \mathbb{R}^{[\leq k]} \).
\end{itemize}

\textbf{(3) Continuum Cardinality:}  
We construct an injective map from \( 2^{\mathbb{N}} \) into \( \mathbb{R}^{\mathbb{F}_\omega} \):

\begin{enumerate}
\item For each infinite binary sequence \( X \subseteq \mathbb{N} \), define a definability chain \( C_X = \{ \mathcal{F}_n^{(X)} \} \) as follows:
\[
\mathcal{F}_1^{(X)} := \mathsf{RCA}_0 + \forall k\, [\varphi_X(k)\downarrow \lor \varphi_X(k)\uparrow],
\]
where \( \varphi_X \) is an \( X \)-computable function defining the characteristic function \( \chi_X \).

\[
\mathcal{F}_{n+1}^{(X)} := \mathcal{F}_n^{(X)} + \mathrm{Con}(\mathcal{F}_n^{(X)}) + \text{“$\emptyset^{(n)}$ exists”}.
\]

\item Define the real number:
\[
r_X := \sum_{k \in X} \frac{1}{2^{k+1}} \in \mathbb{R}_{\mathcal{F}_1^{(X)}} \subseteq \mathbb{R}^{[\leq 1]} \subseteq \mathbb{R}^{\mathbb{F}_\omega}.
\]

\item \emph{Admissibility verification:}
\begin{itemize}
    \item \textbf{Strict growth:} By Gödel's Second Incompleteness Theorem,
    \[
    \mathcal{F}_{n+1}^{(X)} \not\vdash \mathrm{Con}(\mathcal{F}_n^{(X)}).
    \]
    \item \textbf{Effective verification:} The inclusion \( \mathbb{R}_{\mathcal{F}_n^{(X)}} \subsetneq \mathbb{R}_{\mathcal{F}_{n+1}^{(X)}} \) is checkable in \( \mathsf{ACA}_0 \) via:
    \[
    \exists \psi\, [\mathcal{F}_{n+1}^{(X)} \vdash \psi \land \forall \phi\, (\mathcal{F}_n^{(X)} \vdash \phi \Rightarrow \psi \neq \phi)].
    \]
\end{itemize}

\item \emph{Cardinality preservation:}
\begin{itemize}
    \item Different \( X \) yield distinct \( r_X \).
    \item The mapping \( X \mapsto C_X \mapsto r_X \) is injective.
    \item All \( C_X \in \mathbb{F}_\omega \) by admissibility verification above.
\end{itemize}

\item The inequality \( |\mathbb{R}^{\mathbb{F}_\omega}| \leq \mathfrak{c} \) follows from:
\begin{itemize}
    \item Each real is defined by finite syntax in some \( \mathcal{F}_n^{(C)} \).
    \item There are countably many definitions per chain.
    \item The set \( \mathbb{F}_\omega \) has cardinality \( \mathfrak{c} \).
\end{itemize}
Hence \( |\mathbb{R}^{\mathbb{F}_\omega}| \leq \aleph_0 \cdot \mathfrak{c} = \mathfrak{c} \), and we conclude:
\[
|\mathbb{R}^{\mathbb{F}_\omega}| = \mathfrak{c}.
\]
\end{enumerate}
\end{proof}

\section{Stratified Alternative to the Continuum Hypothesis}

The classical Continuum Hypothesis (CH) asks whether there exists a cardinality strictly between \( \aleph_0 \) and \( \mathfrak{c} \). In our stratified framework, this binary perspective is replaced by a transfinite progression of definability thresholds. The continuum no longer appears as a single, structureless entity, but as the limit of a layered process of formal expressibility.

\begin{definition}[Stratified Cardinal Sequence]
For each \( n \in \mathbb{N} \), define:
\[
    \mathbb{R}^{[\leq n]} := \bigcup_{k = 0}^{n} \mathbb{R}^{(k)}, \qquad
    \kappa_n := \left| \mathbb{R}^{[\leq n]} \right|.
\]
This yields a strictly increasing sequence of cardinals:
\[
    \kappa_0 < \kappa_1 < \cdots < \kappa_n < \cdots < \kappa_\omega := \left| \mathbb{R}^{\mathbb{F}_\omega} \right| = \mathfrak{c}.
\]
\end{definition}

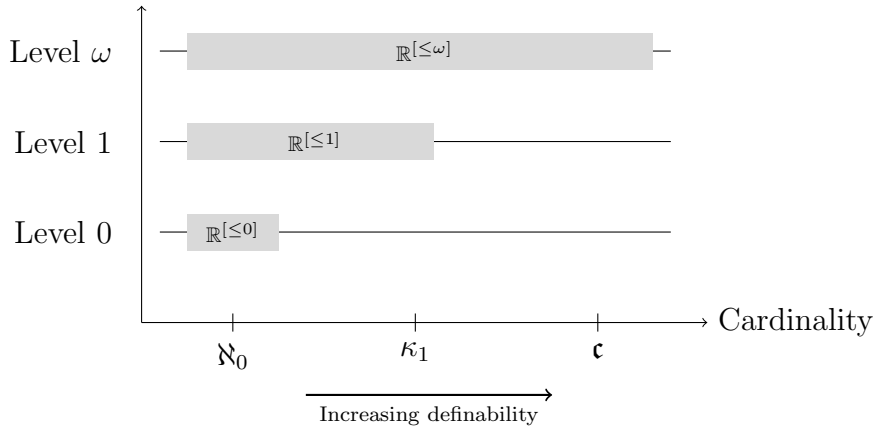
\begin{figure}[ht]
\centering
\begin{tikzpicture}[scale=1.2]
  \draw[->] (0,0) -- (6.2,0) node[right] {Cardinality};
  \draw[->] (0,0) -- (0,3.5);

  \foreach \x/\label in {1/{$\aleph_0$}, 3/{$\kappa_1$}, 5/{$\mathfrak{c}$}} {
    \draw (\x,0.1) -- (\x,-0.1) node[below] {\label};
  }

  \foreach \y/\lvl in {1/0, 2/1, 3/\omega} {
    \draw (0.2,\y) -- (5.8,\y);
    \node[left] at (-0.2,\y) {Level $\lvl$};
  }

  \fill[gray!30] (0.5,0.8) rectangle (1.5,1.2);
  \node at (1.0,1.0) {\scriptsize $\mathbb{R}^{[\leq 0]}$};

  \fill[gray!30] (0.5,1.8) rectangle (3.2,2.2);
  \node at (1.9,2.0) {\scriptsize $\mathbb{R}^{[\leq 1]}$};

  \fill[gray!30] (0.5,2.8) rectangle (5.6,3.2);
  \node at (3.1,3.0) {\scriptsize $\mathbb{R}^{[\leq \omega]}$};

  \draw[->, thick] (1.8,-0.8) -- (4.5,-0.8) node[below, midway] {\scriptsize Increasing definability};
\end{tikzpicture}
\caption{Stratified view of the continuum: each level adds new reals while preserving previous definability classes. The final stage $\mathbb{R}^{[\leq \omega]}$ reaches full cardinality $\mathfrak{c}$.}
\label{fig:cardinal-stratification}
\end{figure}

\begin{theorem}[CH Reinterpreted via Stratification]
\label{thm:stratified-ch}
Within the stratified framework, the classical CH is reinterpreted as the question: 
\[
    \text{“Is there a finite \( n \) such that } \kappa_n = \mathfrak{c} \text{?”}
\]
The answer is negative. The continuum \( \mathfrak{c} \) does not appear at any finite stage but only as the limit of definability layers:
\[
    \lim_{n \to \infty} \kappa_n = \mathfrak{c}, \quad \text{and} \quad \forall n,\ \kappa_n < \mathfrak{c}.
\]
\end{theorem}

\begin{remark}[Examples of Definability Thresholds]
Standard mathematical constants naturally fall into this hierarchy of definability:
\begin{itemize}
    \item \textbf{Level 0:} basic arithmetic numbers such as \( \mathbb{Q} \subset \mathbb{R}^{(0)} \);
    \item \textbf{Level 1:} computable reals, including classical constants like \( \pi \), \( e \), and \( \sqrt{2} \), all of which are Turing-computable;
    \item \textbf{Level \( n \geq 2 \):} non-computable reals such as Chaitin’s \( \Omega \), or reals whose definitions rely on convergence properties provable only in stronger formal systems (e.g., certain Diophantine limits).
\end{itemize}
This illustrates that in stratified analysis, it is definability—not cardinality—that governs the structure of the continuum.
\end{remark}

\begin{example}[Admissible Chain Capturing \( \Omega \)]
\label{ex:omega-chain}
Consider an admissible stratified chain \( \{\mathcal{F}_n\}_{n \in \mathbb{N}} \) where:
\begin{itemize}
    \item Each \( \mathcal{F}_n \) extends \( \mathrm{PA} \) with \( \Sigma^0_n \)-induction and can decide halting problems relative to \( \emptyset^{(n-1)} \);
    \item Theories are strictly increasing: \( \mathcal{F}_n \subsetneq \mathcal{F}_{n+1} \), e.g., via \( \Sigma^0_{n+1} \)-comprehension.
\end{itemize}

Define approximants to Chaitin's \( \Omega \) by:
\begin{enumerate}
    \item Let \( H_n := \{ p \mid \mathcal{F}_n \vdash \text{“$p$ halts in } \leq n \text{ steps”} \} \)
    \item Define \( \Omega_n := \sum_{p \in H_n} 2^{-|p|} \), so that:
    \begin{itemize}
        \item \( \Omega_n \in \mathbb{R}^{(n)} \) and is computable relative to \( \emptyset^{(n-1)} \);
        \item The sequence \( \{ \Omega_n \} \) is monotonic: \( \Omega_n \leq \Omega_{n+1} \);
        \item It converges: \( \lim_{n \to \infty} \Omega_n = \Omega \), since \( \mathbb{R}^{\mathbb{F}_\omega} \) includes all admissible chains of formal expressibility.
    \end{itemize}
\end{enumerate}

This illustrates three key phenomena:
\begin{itemize}
    \item \textbf{Layered Definability}: Each \( \Omega_n \) lies in \( \mathbb{R}^{(n)} \) but not in \( \mathbb{R}^{(n-1)} \);
    \item \textbf{Limit Capture}: \( \Omega \in \mathbb{R}^{\mathbb{F}_\omega} \setminus \bigcup_{n\in\mathbb{N}} \mathbb{R}^{(n)} \);
    \item \textbf{Non-uniformity}: No single \( \mathcal{F}_n \) can prove the value of any specific bit of \( \Omega \).
\end{itemize}
\end{example}

\begin{definition}[Stratified Regularity]
\label{def:stratified-regularity}
Within the framework of stratified definability, a cardinal \( \kappa_n := \left| \mathbb{R}^{[\leq n]} \right| \) is said to be \emph{regular} if it cannot be expressed as a union of fewer than \( \kappa_n \) sets, each of strictly smaller cardinality, drawn from definability layers of rank less than \( n \).

This internal notion parallels the classical concept of regularity but is interpreted constructively through the closure properties of definability strata.
\end{definition}

\begin{example}
Let \( \mathbb{R}^{[\leq n]} \) be the class of reals definable by procedures with at most \( n \) nested applications of the Turing jump. Then the corresponding cardinal \( \kappa_n \) is stratified-regular: no collection of \((n{-}1)\)-definable sets, even taken in totality from preceding layers, can jointly define all elements of \( \mathbb{R}^{[\leq n]} \).
\end{example}

\begin{theorem}[Stratified Density and Limit Structure]
The sequence \( \{ \kappa_n \} \), where \( \kappa_n := \left| \mathbb{R}^{[\leq n]} \right| \), satisfies the following properties within the stratified framework:
\begin{itemize}
    \item Each \( \kappa_n \) is \emph{regular} in the sense of Definition~\ref{def:stratified-regularity}, forming a self-contained definability closure not decomposable via lower layers;

    \item The global continuum \( \mathfrak{c} := \left| \mathbb{R}^{\mathbb{F}_\omega} \right| \) emerges only as the limit of the definability hierarchy:
    \[
        \mathfrak{c} = \lim_{n \to \infty} \kappa_n,
    \]
    where the limit is taken over admissible definability extensions;

    \item There is no finite stage \( n \) such that \( \kappa_n = \mathfrak{c} \): the process of definability never stabilizes at any finite level.
\end{itemize}
\end{theorem}

\begin{remark}[Foundational Implications]
The stratified model replaces the cardinal jump of classical CH with a fine-grained spectrum of constructive expressibility. The continuum becomes not a static totality, but a transfinite unfolding of definitional depth — an infinite ascent through layers of meaning. Real numbers thus acquire internal genealogies, and mathematics becomes a stratified epistemic landscape rather than a Platonic snapshot.
\end{remark}

\subsection*{Cardinality of Incremental Layers}

\noindent\textbf{Convention.}
For notational clarity, we fix the convention:
\[
\Delta_n^{\mathbb{F}_\omega} := \mathbb{R}^{[\leq n]} \setminus \mathbb{R}^{[\leq n-1]}, \quad \text{with } \mathbb{R}^{[\leq -1]} := \emptyset.
\]

\begin{theorem}[Continuum Cardinality of Definability Increments]
\label{thm:delta-n-continuum}
For every fixed \( n \in \mathbb{N} \), the definability increment
\[
\Delta_{n+1}^{\mathbb{F}_\omega} := \mathbb{R}^{[\leq n+1]} \setminus \mathbb{R}^{[\leq n]}
\]
has cardinality \( \mathfrak{c} \). That is,
\[
\left| \Delta_{n+1}^{\mathbb{F}_\omega} \right| = \mathfrak{c}.
\]
\end{theorem}

\begin{proof}
We explicitly construct a continuum-sized family of real numbers, each definable at level \( n+1 \), but at no lower level.

\textbf{Step 1: Definition of the index class.}  
Let \( \mathcal{A} \subseteq \mathcal{P}(\mathbb{N}) \) be the class of all infinite computably enumerable (c.e.) subsets \( A \subseteq \mathbb{N} \) such that:
\begin{itemize}
    \item \( A \) is not definable in any formal system \( \mathcal{F}_k \), for \( k \leq n \);
    \item \( A \) is Turing reducible to \( \emptyset^{(n+1)} \), but not to \( \emptyset^{(n)} \).
\end{itemize}

By the classical jump hierarchy, such sets exist in abundance: their Turing degrees form a perfect set of size \( \mathfrak{c} \) (see e.g.\ \cite{Simpson2009Subsystems}). Thus, \( |\mathcal{A}| = \mathfrak{c} \).

\textbf{Step 2: Encoding real numbers.}  
For each \( A \in \mathcal{A} \), define a real number via binary characteristic encoding:
\[
r_A := \sum_{k \in A} 2^{-k}.
\]
As \( A \) is infinite and not ultimately periodic, each \( r_A \) is irrational and distinct for distinct \( A \).

\textbf{Step 3: Constructing the definability chain.}  
We define a stratified chain \( \{ \mathcal{F}_k^{(A)} \}_{k \in \mathbb{N}} \in \mathbb{F}_\omega \) by:
\[
\mathcal{F}_k^{(A)} := \mathcal{F}_k \text{ for } k \leq n, \quad
\mathcal{F}_{n+1}^{(A)} := \mathcal{F}_n + \phi_A,
\]
where \( \phi_A \) is the following computably enumerable axiom scheme:
\[
\phi_A := \left\{ \text{``Bit } d_k \text{ of } r \text{ equals } 1 \iff k \in A \text{''} \right\}_{k \in \mathbb{N}}.
\]
That is, the theory \( \mathcal{F}_{n+1}^{(A)} \) explicitly asserts the characteristic function of \( A \), allowing it to define the binary expansion of \( r_A \).

\textbf{Step 4: Definability at level \( n+1 \).}  
In \( \mathcal{F}_{n+1}^{(A)} \), the Cauchy sequence
\[
q_m := \sum_{\substack{k \in A \\ k \leq m}} 2^{-k}
\]
is provably convergent, with a computable modulus \( \mu(\varepsilon) := \min\{k \in \mathbb{N} \mid 2^{-k} < \varepsilon\} \), since the bits are explicitly fixed. Hence:
\[
r_A \in \mathbb{R}_{\mathcal{F}_{n+1}^{(A)}} \subseteq \mathbb{R}^{[\leq n+1]}.
\]

\textbf{Step 5: Non-definability below level \( n+1 \).}  
Suppose for contradiction that \( r_A \in \mathbb{R}^{[\leq n]} \). Then there exists some \( k \leq n \) such that \( \mathcal{F}_k \vdash \text{``}\{q_m\} \to r_A \text{ with modulus} \text{''} \). But any such proof must define the characteristic function of \( A \), which contradicts \( A \notin \mathrm{Def}(\mathcal{F}_k) \). Therefore:
\[
r_A \notin \mathbb{R}^{[\leq n]}.
\]

\textbf{Step 6: Injectivity and cardinality.}  
The map \( A \mapsto r_A \) is injective, as binary representations are distinct. Hence,
\[
\left| \left\{ r_A \mid A \in \mathcal{A} \right\} \right| = \mathfrak{c},
\quad \text{and} \quad
\left\{ r_A \right\} \subseteq \Delta_{n+1}^{\mathbb{F}_\omega}.
\]

\textbf{Step 7: Upper bound.}  
Since \( \Delta_{n+1}^{\mathbb{F}_\omega} \subseteq \mathbb{R}^{\mathbb{F}_\omega} \), and \( |\mathbb{R}^{\mathbb{F}_\omega}| = \mathfrak{c} \), we conclude:
\[
\left| \Delta_{n+1}^{\mathbb{F}_\omega} \right| = \mathfrak{c}.
\]
\end{proof}

\begin{theorem}[Self-Density of Definability Layers]
For every \( n \in \mathbb{N} \), the set \( \Delta_n^{\mathbb{F}_\omega} \) is dense in itself:
\[
\forall r \in \Delta_n^{\mathbb{F}_\omega},\ \forall \varepsilon > 0,\ \exists s \in \Delta_n^{\mathbb{F}_\omega}\setminus\{r\} \text{ such that } |r - s| < \varepsilon.
\]
\end{theorem}

\begin{proof}[Sketch of Proof]
Fix \( n \in \mathbb{N} \). Let \( r \in \Delta_n^{\mathbb{F}_\omega} \), so there exists an admissible definability chain \( \{\mathcal{F}_k\} \in \mathbb{F}_\omega \) such that:
\begin{itemize}
    \item \( r \in \mathbb{R}_{\mathcal{F}_n} \);
    \item \( r \notin \bigcup_{k < n} \mathbb{R}_{\mathcal{F}_k} \).
\end{itemize}

By definition of \( \mathbb{R}_{\mathcal{F}_n} \), there exists a rational sequence \( \{q_k\} \subset \mathbb{Q} \) and a provable modulus of convergence \( m: \mathbb{N} \to \mathbb{N} \) such that \( \mathcal{F}_n \vdash \text{“}\lim q_k = r \text{”} \).

Let \( \varepsilon > 0 \) be arbitrary. Choose an index \( N \in \mathbb{N} \) such that:
\[
\forall k, \ell \geq N,\quad |q_k - q_\ell| < \varepsilon/2,
\]
and fix some \( k > N \). Then define the perturbed rational sequence \( \{q_k'\} \) as:
\[
q_i' := 
\begin{cases}
q_i, & i \neq k, \\
q_k + \delta, & i = k,
\end{cases}
\]
where \( \delta \in \mathbb{Q} \) is a small rational with \( 0 < |\delta| < \varepsilon/4 \), chosen such that the limit \( r' := \lim q_k' \) differs from \( r \), i.e., \( r' \neq r \).

Because \( \delta \) is small and appears only at a single index, the convergence properties of the modified sequence remain provable in \( \mathcal{F}_n \). Hence:
\[
r' \in \mathbb{R}_{\mathcal{F}_n}.
\]

Moreover, since the perturbation is minor and not definable in any weaker system \( \mathcal{F}_k \), \( k < n \), the new real \( r' \notin \bigcup_{k < n} \mathbb{R}_{\mathcal{F}_k} \), and therefore:
\[
r' \in \Delta_n^{\mathbb{F}_\omega}, \quad r' \neq r, \quad \text{and} \quad |r - r'| < \varepsilon.
\]

Thus, for any \( r \in \Delta_n^{\mathbb{F}_\omega} \) and any \( \varepsilon > 0 \), there exists \( s := r' \in \Delta_n^{\mathbb{F}_\omega} \setminus \{r\} \) such that \( |r - s| < \varepsilon \), which proves that \( \Delta_n^{\mathbb{F}_\omega} \) is dense in itself.
\end{proof}

\paragraph{Stratified Separation of the Continuum.}
Unlike the classical continuum \( \mathbb{R} \), which appears as an undifferentiated totality of cardinality \( \mathfrak{c} \), the stratified model \( \mathbb{R}^{\mathbb{F}_\omega} \) permits a countable partition into definability increments:
\[
\mathbb{R}^{\mathbb{F}_\omega} = \bigcup_{n=0}^\infty \Delta_n^{\mathbb{F}_\omega}, \qquad \text{where } |\Delta_n^{\mathbb{F}_\omega}| = \mathfrak{c},\quad \Delta_n^{\mathbb{F}_\omega} \cap \Delta_m^{\mathbb{F}_\omega} = \emptyset \text{ for } n \neq m.
\]
Each \( \Delta_n^{\mathbb{F}_\omega} \) is not only uncountable but also dense-in-itself, forming an internally coherent continuum layer. Such a partition into disjoint dense subsets of continuum size is impossible in classical set theory, where any two dense \( \mathfrak{c} \)-sized subsets must intersect. Thus, the stratified framework enables a constructive decomposition of the continuum inaccessible to classical models.

\section{From Fractal Continuum to Classical Completion}

\subsection{Constructive Projection onto the Classical Interval}

The stratified architecture of \( \mathbb{R}^{\mathbb{F}_\omega} \) induces a canonical decomposition of the unit interval into distinct layers of definability complexity. This construction preserves the topological structure of \([0,1]\) while revealing its internal hierarchical organization.

\begin{definition}[Fractal Unit Interval]
The projection of the fractal continuum onto \([0,1]\) is defined as:
\[
[0,1]^{\mathbb{F}_\omega} := \bigsqcup_{n=0}^\infty \Delta_n^{[0,1]},
\]
where each definability layer is given by:
\[
\Delta_n^{[0,1]} := \Delta_n^{\mathbb{F}_\omega} \cap [0,1],
\]
with \( \mathbb{R}^{[\leq n]} \) denoting the reals constructively definable at level \( \leq n \).
\end{definition}

\begin{proposition}[Metamathematical Topology]
Each stratum \( \Delta_n^{[0,1]} \) satisfies:
\begin{itemize}
    \item \textbf{Cardinality:} \( |\Delta_n^{[0,1]}| = \mathfrak{c} \), witnessed by the injective coding:
    \[
    X \mapsto r_X := \sum_{k \in X} 2^{-k}, \quad r_X \in \Delta_n^{[0,1]},
    \]
    for appropriately chosen oracle sets \( X \).

    \item \textbf{Density:} For any \( x < y \in [0,1] \), the existence of some \( r \in \Delta_n^{[0,1]} \cap (x, y) \) is provable in the meta-theory, and such \( r \) is definable in \( \mathcal{F}_{n+1} \).

    \item \textbf{Metrical Disjointness:}
    \[
    \forall n \neq m,\quad \Delta_n^{[0,1]} \cap \Delta_m^{[0,1]} = \emptyset, \quad \text{and} \quad \bigcup_{n} \Delta_n^{[0,1]} = [0,1]^{\mathbb{F}_\omega}.
    \]
\end{itemize}
\end{proposition}

\begin{theorem}[Definable vs Classical Continuum]
The decomposition \( \{ \Delta_n^{[0,1]} \}_{n \in \mathbb{N}} \) reveals a structural contrast:
\begin{itemize}
    \item In the \emph{fractal model}:
    \begin{itemize}
        \item All layers exist constructively without invoking choice
        \item Each real is effectively indexed by its definability level
    \end{itemize}
    
    \item In \emph{ZF set theory without choice}:
    \begin{itemize}
        \item No Borel or Lebesgue-measurable separation of dense \( \mathfrak{c} \)-subsets into disjoint layers is possible
        \item The definability grading is unobservable without external meta-theory
    \end{itemize}
\end{itemize}
\end{theorem}

\begin{remark}[Philosophical Interpretation]
This framework illustrates that:
\begin{itemize}
    \item \emph{Constructive stratification} resolves the classical tension between:
    \begin{itemize}
        \item Topological connectedness
        \item Set-theoretic decomposability
    \end{itemize}

    \item The continuum can be seen as structured by \emph{degrees of definability}, not just cardinality

    \item Measure-theoretic intuition survives without non-constructive selection, grounded instead in logical complexity
\end{itemize}
\end{remark}

\begin{table}[h]
\centering
\begin{tabular}{@{}lcc@{}}
\toprule
\textbf{Property} & \textbf{Classical \([0,1]\)} & \textbf{Fractal \([0,1]^{\mathbb{F}_\omega}\)} \\
\midrule
Decomposition & Non-constructive & Definable layers \\
Cardinality & \( \mathfrak{c} \) (undifferentiated) & \( \mathfrak{c} \) (stratified) \\
Density & Homogeneous & Layered and dense \\
Ontology & Platonic & Constructive-processual \\
\bottomrule
\end{tabular}
\caption{Comparison of interval conceptions in classical and stratified models}
\label{tab:interval-comparison}
\end{table}

\subsection{Constructive Synthesis of the Classical Continuum}

The classical real interval \([0,1]\) can be viewed as the completion of a stratified constructive core by adjoining non-constructive elements arising from choice-dependent principles. This yields a unified semantic model blending definability layers with classical totality.

\begin{theorem}[Fractal–Classical Synthesis]
Let \( \mathbb{R}^{\mathbb{F}_\omega} \) denote the stratified fractal continuum, and let \( \mathbb{R}^{\neg\mathrm{con}} \) denote the complement of constructively definable reals, i.e., the class of non-constructive reals arising from principles such as the Axiom of Choice. Then:
\[
[0,1]_{\mathrm{classic}} \cong [0,1]^{\mathbb{F}_\omega} \oplus \mathbb{R}^{\neg\mathrm{con}} \big/ \sim,
\]
where \( \sim \) identifies:
\begin{itemize}
    \item Constructive reals \( r \in \mathbb{R}^{\mathbb{F}_\omega} \) with their classical representations;
    \item Non-constructive reals \( x \in \mathbb{R}^{\neg\mathrm{con}} \) with ideal points in the classical model.
\end{itemize}
\end{theorem}

\begin{proof}[Proof sketch]
We construct the synthesis in three steps:
\begin{enumerate}
\item \textbf{Embedding of the fractal continuum:}
\[
\iota : [0,1]^{\mathbb{F}_\omega} \hookrightarrow [0,1]_{\mathrm{classic}}, \quad \iota(r) = r,
\]
is valid, as all constructively definable reals exist in ZFC.

\item \textbf{Adjoining non-constructive reals:}  
For every \( x \in [0,1] \setminus \iota([0,1]^{\mathbb{F}_\omega}) \), define an ideal representative:
\[
\hat{x} := \mathrm{Choice}\left( \{ A \subseteq \mathbb{N} \mid x_A \approx x \} \right),
\]
where \( x_A \) is the real number defined by the binary expansion induced by \( A \).

\item \textbf{Properties of the resulting space:}
\begin{itemize}
    \item \emph{Completeness:} Follows from Dedekind or Cauchy completion of the dense set \( \mathbb{R}^{\mathbb{F}_\omega} \);
    \item \emph{Cardinality:} The complement \( \mathbb{R}^{\neg\mathrm{con}} \) has cardinality \( \mathfrak{c} \);
    \item \emph{Density:} \( \mathbb{R}^{\mathbb{F}_\omega} \) is dense in \( [0,1] \), both topologically and approximatively.
\end{itemize}
\end{enumerate}
\end{proof}

\begin{table}[h]
\centering
\renewcommand{\arraystretch}{1.25}
\begin{tabular}{@{}|l|c|c|@{}}
\hline
\textbf{Property} & \textbf{Fractal Core} & \textbf{Non-Constructive Completion} \\ \hline
Definability & Explicitly syntactic & Implicit (choice-based) \\
Measurability & Effectively approximable & Abstractly integrable \\
Ontological Status & Constructive / procedural & Platonic / ideal \\ \hline
\end{tabular}
\caption{Semantic layers of the classical continuum}
\label{tab:semantic-continuum}
\end{table}

\begin{proposition}[Meta-Theoretic Consequences]
\leavevmode
\begin{itemize}
    \item \textbf{Model-theoretic:}
    \[
    \mathrm{Th}([0,1]_{\mathrm{classic}}) = \mathrm{Th}([0,1]^{\mathbb{F}_\omega}) \cup \{ \exists x\, \neg\mathrm{ConDef}(x) \}.
    \]

    \item \textbf{Topological:} \( \mathbb{R}^{\mathbb{F}_\omega} \) forms a dense substructure of \( [0,1] \); non-constructive points are non-isolated.

    \item \textbf{Computational:} Every classical real \( x \in [0,1] \) is a limit
    \[
    x = \lim_{n \to \infty} r_n, \quad r_n \in \mathbb{R}^{\mathbb{F}_\omega},
    \]
    but for non-constructive \( x \), no computable sequence \( \{r_n\} \) converges to it.
\end{itemize}
\end{proposition}

\begin{remark}[Unified Continuum Model]
This synthesis provides a layered semantics for the real continuum:
\begin{itemize}
    \item Constructive strata give fine-grained control over definability and complexity;
    \item Non-constructive elements allow for classical totality and abstraction;
    \item The union forms a hybrid model suitable for both foundational analysis and practical mathematics.
\end{itemize}
\end{remark}

\subsection{Reformulating Cardinality and the Continuum Hypothesis}

The classical notion of cardinality, grounded in set-theoretic bijections, fails to capture the internal structure of the continuum as revealed by definability stratification. In the stratified framework, all levels \( \mathbb{R}^{[\leq n]} \) share cardinality \( \mathfrak{c} \), yet remain fundamentally distinct in their syntactic and semantic complexity.

\begin{principle}[Stratified Intermediate Complexity in the Continuum]
The stratified continuum \( \mathbb{R}^{\mathbb{F}_\omega} \) exhibits a canonical hierarchy
\[
\mathbb{R}^{[\leq 0]} \subsetneq \mathbb{R}^{[\leq 1]} \subsetneq \cdots \subsetneq \mathbb{R}^{\mathbb{F}_\omega},
\]
such that each inclusion reflects an essential increase in definability strength.
\end{principle}

\begin{principle}[Stratified Continuum Principle (SCP)]
The stratified continuum \( \mathbb{R}^{\mathbb{F}_\omega} \) admits no internal collapse: each level \( \mathbb{R}^{[\leq n]} \) under any transformation preserving definability rank or logical complexity.
\end{principle}

\begin{remark}[Against Classical Cardinality]
This hierarchy exposes the limitations of Cantorian cardinality. While all sets \( \mathbb{R}^{[\leq n]} \) are equinumerous in \( \mathsf{ZFC} \), they are not structurally equivalent within any constructive or syntactically grounded model. The traditional notion of “size” loses descriptive power when internal logical complexity becomes the primary organizing principle.

Accordingly, we interpret intermediate “cardinalities” not as cardinal numbers between \( \aleph_0 \) and \( \mathfrak{c} \), but as definability degrees. Each level:
\begin{itemize}
    \item arises from a strictly stronger formal system;
    \item contains syntactically irreducible elements;
    \item and supports a richer set of constructive operations.
\end{itemize}
\end{remark}

\begin{table}[ht]
\centering
\renewcommand{\arraystretch}{1.2}
\begin{tabular}{@{}|l|c|c|@{}}
\hline
\textbf{Aspect} & \textbf{Classical View (ZFC)} & \textbf{Stratified View} \\ \hline
Cardinality & \( \aleph_0 < \mathfrak{c} \), binary & All levels \( \mathfrak{c} \), hierarchically distinct \\
Distinction & By bijection & By definability strength \\
Collapse & CH unresolved & SCP explicitly false \\
Continuum & Homogeneous set & Stratified definability structure \\
\hline
\end{tabular}
\caption{Contrasting classical and stratified views of cardinality}
\label{tab:cardinality-reform}
\end{table}

\section{Stratified Definability Compression: A Constructive Alternative to Kolmogorov Complexity}

In the framework of the fractal continuum \( \mathbb{R}^{\mathbb{F}_\omega} \), each real number arises not as an absolute set-theoretic object, but as a definable limit within a particular formal system \( \mathcal{F}_n \). This gives rise to a novel form of semantic compression: rather than minimizing program length (as in classical Kolmogorov complexity), we minimize the \emph{definability level} and the size of the syntactic witness required to construct a number. We refer to this phenomenon as \emph{stratified definability compression}.

\subsection*{Core Idea}

Let \( r \in \mathbb{R} \). We do not describe it as a completed decimal expansion or a bitstream. Instead, we find the smallest formal system \( \mathcal{F}_n \) in which there exists:
\begin{itemize}
    \item A rational Cauchy sequence \( \{ q_k \} \) with provable modulus of convergence \( m(k) \), and
    \item A proof within \( \mathcal{F}_n \) that \( \{ q_k \} \to r \).
\end{itemize}
This process compresses an infinite object into a finite syntactic datum within a bounded definability level.

\begin{definition}
For \( r \in \mathbb{R}^{\mathbb{F}_\omega} \), the \emph{definability compression} \( \mathrm{DComp}(r) \) is the lexicographically minimal pair \( (n, \sigma) \) such that:
\begin{enumerate}
    \item \( r \in \mathbb{R}_{\mathcal{F}_n} \),
    \item \( \sigma \) is a syntactic derivation in \( \mathcal{F}_n \) proving that a rational Cauchy sequence \( \{ q_k \} \to r \) with a definable modulus \( m(k) \),
    \item For all \( m < n \), no such derivation exists in \( \mathcal{F}_m \),
    \item Among all such derivations in \( \mathcal{F}_n \), \( \sigma \) has minimal Gödel number (i.e., is the shortest syntactic representation under a fixed enumeration; not a literal encoding, but a canonical representative of an equivalence class of derivations).
\end{enumerate}
\end{definition}

\subsection*{Illustrative Examples}

\paragraph{Example 1: \( \pi \) (Analytical constant in \( \mathsf{ACA}_0 \)).}
The Leibniz series
\[
\pi = 4 \sum_{k=0}^{\infty} \frac{(-1)^k}{2k+1}
\]
is provably convergent in \( \mathsf{ACA}_0 \). Therefore, \( \pi \in \mathbb{R}_{\mathcal{F}_n} \) for some \( \mathcal{F}_n \supseteq \mathsf{ACA}_0 \), and its definability compression is
\[
\mathrm{DComp}(\pi) = \left(n, \text{"Leibniz series + convergence proof in } \mathcal{F}_n \text{"} \right).
\]

\paragraph{Example 2: Summation reals (Classification preview).}
Let \( r_X = \sum_{k \in X} 2^{-k} \) for some \( X \subseteq \mathbb{N} \). The definability compression of \( r_X \) depends on the \emph{logical complexity} of \( X \). A full classification is provided below.

\paragraph{Example 3: Ackermann-type reals (Computable, fast-growing).}
Let \\\( r_A = \sum_{k=0}^\infty \frac{f_A(k)}{2^{k!}} \), where \( f_A \) is the Ackermann function. Then:
\begin{itemize}
    \item \( r_A \in \mathbb{R}_{\mathcal{F}_n} \) for some \( n \), since \( f_A \) is total and computable;
    \item \( \mathrm{DComp}(r_A) \) requires \( \mathcal{F}_n \supseteq \mathsf{PRA} \) (or \( \mathsf{I\Sigma}_1 \)), as the totality of \( f_A \) is unprovable in weaker systems;
    \item The syntactic witness \( \sigma \) must encode a formal proof of \( f_A \)'s totality.
\end{itemize}

\paragraph{Example 4: Specker sequences (Computable sequence, non-computable limit).}
Let \( \{ r_s \} \) be a computable, bounded sequence with no computable limit (a Specker sequence). Then:
\begin{itemize}
    \item Each \( r_s \in \mathbb{R}_{\mathcal{F}_1} \), by computable generation;
    \item The classical limit \( \lim r_s \) exists \emph{as a real number}, but is provably non-computable;
    \item The limit belongs to \( \mathbb{R}_{\mathcal{F}_n} \) only for \( \mathcal{F}_n \vdash \mathsf{ACA}_0 \), and \( \mathrm{DComp}(\lim r_s) \) reflects this necessity.
\end{itemize}

\paragraph{Example 5: Hyperarithmetical reals (Limit of arithmetical expressibility).}
Let \( r_H = \sum_{k \in \mathcal{O}} 2^{-k} \), where \( \mathcal{O} \) is Kleene’s system of notations for recursive ordinals. Then:
\begin{itemize}
    \item \( r_H \notin \mathbb{R}_{\mathcal{F}_n} \) for any arithmetical \( \mathcal{F}_n \) (i.e., \( n < \omega \));
    \item \( \mathrm{DComp}(r_H) \) exists only in systems extending \( \mathsf{ATR}_0 \);
    \item The syntactic witness \( \sigma \) must encode the well-foundedness of \( \mathcal{O} \).
\end{itemize}

\paragraph{Example 6: Martin-Löf random reals (Uncompressible information).}
For \\Chaitin’s halting probability \( \Omega = 0.b_1b_2b_3\dots \), interpreted as a real:
\begin{itemize}
    \item \( \Omega \notin \mathbb{R}^{\mathcal{F}_\omega} \), since it admits no syntactic witness;
    \item However, truncations \( \Omega_n = 0.b_1\dots b_n \) are rational and satisfy \( \mathrm{DComp}(\Omega_n) = (n, \sigma_n) \) with definability level \( n \to \infty \);
    \item The growth rate \( n(\Omega_n) \) reflects the randomness deficiency of \( \Omega \).
\end{itemize}

\paragraph{Example 7: Fine-structure constant \( \alpha \) (Physical definability).}
Assuming \( \alpha \) is derivable from first principles in physics:
\begin{itemize}
    \item \( \mathrm{DComp}(\alpha) \) encodes the minimal formal system sufficient to express its derivation;
    \item If derivable within formalizable QED, then \( \mathcal{F}_n \vdash \mathsf{ACA}_0 \);
    \item If requiring quantum gravity or beyond, then \( \alpha \notin \mathbb{R}_{\mathcal{F}_n} \) for any currently known \( \mathcal{F}_n \).
\end{itemize}

\paragraph{Example 8: Feigenbaum constants (Universality in chaos).}
Let \( \delta, \alpha \) denote the Feigenbaum bifurcation constants in nonlinear dynamics. Then:
\begin{itemize}
    \item \( \delta, \alpha \in \mathbb{R}_{\mathcal{F}_n} \) for some \( \mathcal{F}_n \supseteq \mathsf{ATR}_0 \);
    \item The syntactic witness \( \sigma \) must encode the \emph{existence} of a universality proof via renormalization theory;
    \item The compression depth reflects the constructive complexity of the transition to chaos.
\end{itemize}

\paragraph{Example 9: \( \sqrt{2} \) (Elementary algebraic number).}
Let \( r = \sqrt{2} \). Then:
\begin{itemize}
    \item \( r \in \mathbb{R}_{\mathcal{F}_1} \), since its minimal polynomial and convergence of Newton’s method are provable in weak systems;
    \item \( \mathrm{DComp}(\sqrt{2}) = (1, \sigma) \), where \( \sigma \) proves the convergence of a rational Cauchy sequence (e.g. Heron’s method);
    \item Contrasts with \( \pi \), which requires higher compression due to analytic convergence.
\end{itemize}

\paragraph{Example 10: \( \sin\left(\frac{\sqrt{2}}{6}\right) \) (Hierarchy preservation).}
Let \( r = \sin\left(\frac{\sqrt{2}}{6} \right) \). Then:
\begin{itemize}
    \item \( \sqrt{2} \in \mathbb{R}_{\mathcal{F}_1} \) (via Heron's method in \( \mathsf{RCA}_0 \)), and rational scaling preserves \\\( \mathcal{F}_1 \)-definability;
    \item The Taylor series for sine converges computably at computable inputs, with proof in \( \mathsf{RCA}_0 \);
    \item Thus, \( \mathrm{DComp}(r) = (1, \sigma) \), where \( \sigma \) combines:
    \begin{enumerate}
        \item The \( \mathcal{F}_1 \)-derivation of \( \sqrt{2} \),
        \item Rational scaling to \( \sqrt{2}/6 \),
        \item Taylor approximation proof for \( \sin(x) \);
    \end{enumerate}
    \item \emph{Note:} Non-series representations (e.g. continued fractions) may require \( \mathsf{ACA}_0 \) to prove equivalence.
\end{itemize}

\paragraph{Remark.}
These examples span a hierarchy from elementary computable numbers (\( \sqrt{2} \)), through arithmetically definable constants (\( \pi \)), fast-growing computable sequences (\( r_A \)), limits of computable sequences (\( \lim r_s \)), hyperarithmetical objects (\( r_H \)), and non-definable reals (\( \Omega \)), up to physically defined constants (\( \alpha \)). This illustrates how \( \mathrm{DComp}(r) \) captures both logical and epistemic depth.

\subsection*{Contrast with Classical Compression}

\vspace{0.5em}
\begin{adjustbox}{max width=\textwidth}
\begin{tabular}{@{}|p{4.2cm}|p{5.2cm}|p{6cm}|@{}}
\hline
\textbf{Aspect} & \textbf{Kolmogorov Complexity} & \textbf{Stratified Definability Compression} \\
\hline
Language & Fixed universal machine & Hierarchy of formal systems \( \mathcal{F}_n \) \\
\hline
Description form & Program on Turing machine & Syntactic construction within \( \mathcal{F}_n \) \\
\hline
Compression goal & Shortest program & Minimal definability level \( n \) and minimal derivation \( \sigma \) \\
\hline
What is compressed & Bit sequence / string & Constructive emergence of a number \\
\hline
Measure of complexity & Program length & Logical depth of expressibility \\
\hline
Interpretation & Computational information & Formal definability and epistemic ancestry \\
\hline
Ontology & Platonic bitstrings & Process-relative emergence \\
\hline
\end{tabular}
\end{adjustbox}
\vspace{0.5em}

\subsection*{Computability Limits of Definability Compression}

\begin{lemma}[Compression Complexity of $\emptyset'$-Summation]
Let \( r := \sum_{k \in \emptyset'} 2^{-k} \). Then computing \( \mathrm{DComp}(r) \) is \( \Sigma^0_1 \)-complete:
\begin{itemize}
    \item Determining whether \( r \in \mathbb{R}_{\mathcal{F}_n} \) for a given \( n \) is \( \Sigma^0_1 \)-hard,
    \item No algorithm relative to any \( \mathcal{F}_n \) can compute the minimal pair \( (n, \sigma) \).
\end{itemize}
\end{lemma}

\begin{corollary}
There exists \( r \in \mathbb{R}^{\mathbb{F}_\omega} \) such that the compression degree of \( \mathrm{DComp}(r) \) equals the Turing degree of \( \emptyset' \).
\end{corollary}

\begin{example}[Definability Compression of Summation Reals]
Let \( r_X = \sum_{k \in X} 2^{-k} \) for various \( X \subseteq \mathbb{N} \):
\begin{itemize}
    \item If \( X \) is recursive, then \( \mathrm{DComp}(r_X) = (n, \sigma) \) for small \( n \) and short derivation \( \sigma \).
    \item If \( X \) is \( \Sigma^0_1 \)-complete, then \( n \) is unbounded and \( \sigma \) grows accordingly.
    \item If \( X \) is non-definable (e.g. a random oracle), then \( r_X \notin \mathbb{R}^{\mathbb{F}_\omega} \), and no \( \mathrm{DComp}(r_X) \) exists.
\end{itemize}
\end{example}

\subsection*{Connection to Reverse Mathematics}

The definability compression hierarchy aligns closely with the structure of reverse mathematics:
\begin{itemize}
    \item \( \mathsf{RCA}_0 \): base system for computable numbers and basic sequences,
    \item \( \mathsf{ACA}_0 \): definability of analytical constants (e.g. \( \pi, e \)),
    \item \( \mathsf{ATR}_0 \): transfinite constructions requiring well-founded recursion.
\end{itemize}

Each subsystem corresponds to a definability stratum \( \mathcal{F}_n \), and compression encodes the minimal system needed to express a given number.

\subsection*{Visualization}

\begin{center}
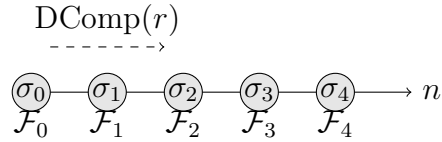

\begin{tikzpicture}
\draw[->, dashed] (0.25,0.6) -- (1.75,0.6) node[midway, above] {\( \mathrm{DComp}(r) \)};
\draw[->] (0,0) -- (5,0) node[right] {\( n \)};
\foreach \n in {0,...,4} {
    \draw (\n,0.1) -- (\n,-0.1) node[below] {\( \mathcal{F}_{\n} \)};
    \draw[fill=gray!20] (\n,0) circle (0.25) node {\( \sigma_{\n} \)};
}
\end{tikzpicture}
\captionof{figure}{Visualization of definability compression: each real number \( r \) becomes expressible at some minimal level \( \mathcal{F}_n \), with corresponding witness \( \sigma_n \).}
\end{center}

\subsection*{Practical Interpretation of Definability Compression}

While \( \mathrm{DComp}(r) = (n, \sigma) \) is formally defined as a minimal definability pair within a stratified hierarchy, its interpretation extends beyond purely metamathematical considerations. The concept serves as a measure of how epistemically deep a real number is within a constructive framework. Several practical domains benefit from this perspective:

\begin{itemize}
    \item \textbf{Proof verification.} In formal proof systems (e.g., Lean, Coq), one often needs to determine the minimal logical base in which a mathematical constant or construction is expressible. The value \( n \) indicates the weakest theory \( \mathcal{F}_n \) sufficient to verify the existence of \( r \), while \( \sigma \) guides the structure of the formal derivation.
    
    \item \textbf{Reverse mathematics.} The classification of theorems by the subsystems required for their proof aligns naturally with definability compression: a number's emergence at level \( \mathcal{F}_n \) reflects its logical dependencies. Weak systems (e.g., \( \mathsf{RCA}_0 \)) cover computable reals; stronger ones (e.g., \( \mathsf{ACA}_0 \), \( \mathsf{ATR}_0 \)) are needed for analytical or transfinite constructions.
    
    \item \textbf{Constructive analysis.} In computational mathematics, it is not enough to assert that a number is computable. One often needs to know the complexity of constructing it within a formal system, especially in resource-bounded or predicative settings. \( \mathrm{DComp}(r) \) acts as a stratified substitute for classical computability claims.
    
    \item \textbf{Proof mining and logic-based complexity.} In bounded arithmetic and logical analysis of algorithms, the derivation \( \sigma \) encodes not just existence but complexity: length of proof, recursion depth, and definability bounds. Compression in this sense becomes a proxy for logical cost.
    
    \item \textbf{Scientific methodology.} In empirical contexts, such as theoretical physics, constants are often known numerically but not formally derived. Definability compression provides a framework for asking: how far up the logical tower must we climb to construct a given constant from first principles?
\end{itemize}

In each case, \( \mathrm{DComp}(r) \) does not measure the raw size or bit-length of \( r \), but its formal position within a layered constructive universe. This enriches our understanding of what it means for a number to be "known," "constructed," or "provable," depending on the logical resources required.

\subsection*{Interpretive Summary}

Stratified definability compression shifts the focus from externally measurable complexity (e.g., Kolmogorov length) to internal epistemic depth. It captures how deep one must go in the tower of definability to access a given mathematical object. In this sense, compression becomes not just a matter of representation, but of logical visibility within a layered constructive universe.

\begin{remark}[Topos-theoretic Interpretation]
The definability compression hierarchy induces a canonical filtration of the real number object \( \mathbb{R} \) in the topos \( \mathbf{Sh}(\mathcal{S}) \), where each layer corresponds to a definability stratum \( \mathcal{F}_n \). The value \( \mathrm{DComp}(r) = (n, \sigma) \) selects the minimal open context in which \( r \) exists as a global element. In this sense, \( \mathbf{Sh}(\mathcal{S}) \) internalizes definability compression as a geometric filtration by expressibility.
\end{remark}

\section{Conclusion: The Fractal Continuum as a Constructive Alternative}

Our investigation reveals a profound shift in the understanding of the continuum. Instead of treating \( \mathbb{R} \) as a static, complete totality endowed with cardinality \( \mathfrak{c} \), we model it as a dynamic, stratified process of definitional emergence. This shift leads to a constructive reinterpretation of real numbers and the continuum itself.

\subsection*{Two Paradigms of Definability}

We distinguished two fundamental paradigms:

\begin{itemize}
    \item \textbf{Non-admissible chains}, which may exhibit collapse or instability, illustrate the fragility of definability in the absence of structural monotonicity. Theorem~\ref{thm:noncollapse} exemplifies this possibility.
    
    \item \textbf{Admissible chains}, defined by strictly increasing definitional power, form the foundation of our framework. Here, definability accumulates coherently, and no real number ever "disappears" once defined.
\end{itemize}

While non-admissible constructions are excluded from the main model, they serve as conceptual counterpoints, showing what must be ruled out for a stable continuum to emerge.

\subsection*{Three Principles of the Constructive Continuum}

The fractal continuum arises from three interlocking principles:

\begin{enumerate}
    \item \textbf{Stratification:} Each definability level \( \mathbb{R}^{[\leq n]} \) represents a cumulative stage of epistemic growth, with
    \[
        \mathbb{R}_{\mathcal{F}_n} \subsetneq \mathbb{R}_{\mathcal{F}_{n+1}}.
    \]

    \item \textbf{Stability:} Within any admissible chain, definability is persistent:
    \[
        r \in \mathbb{R}_{\mathcal{F}_k} \Rightarrow r \in \mathbb{R}_{\mathcal{F}_m} \quad \text{for all } m \geq k.
    \]

    \item \textbf{Completeness:} The totality of constructively definable reals is captured by:
    \[
        \mathbb{R}^{\mathbb{F}_\omega} = \bigcup_{n \in \mathbb{N}} \mathbb{R}^{[\leq n]}.
    \]
\end{enumerate}

Together, these principles define a continuum that is epistemically traceable and formally stratified, yet ultimately coextensive in cardinality with the classical \( \mathbb{R} \).

\subsection*{From Cardinality to Expressibility}

Unlike classical models, which measure sets by abstract cardinality, our framework interprets the continuum through \emph{epistemic accessibility}. A real number exists not by fiat, but through definitional reachability within some admissible path. This process-oriented perspective reveals that many classically trivial reals (e.g., computable sums) lie beyond the provability strength of weak systems such as \( \mathsf{RCA}_0 \), unless supplemented by specific convergence axioms.

Thus, definability becomes sensitive to formal context. The constructive continuum reflects a universe where existence is relative to formal expressibility — not merely syntactic form or external set membership.

\subsection*{Beyond the Continuum Hypothesis}

In this model, the Continuum Hypothesis (CH) loses its applicability. Since there is no canonical, monolithic set \( \mathbb{R} \), but rather a hierarchy of definability layers \( \mathbb{R}^{[\leq n]} \), the classical question “Is there a set of cardinality strictly between \( \aleph_0 \) and \( \mathfrak{c} \)?” dissolves. Each layer \( \mathbb{R}^{[\leq n]} \) is uncountable, yet constructively reachable — a kind of \emph{intermediate continuum}.

This stratification yields a natural family of CH-alternatives: the layers are all equicardinal with \( \mathbb{R} \), but internally ordered by definitional depth. The real continuum becomes not a point of cardinal abstraction, but a spectrum of expressible structures.

\subsection*{The Continuum as a Stratified Construction}

In conclusion, the classical continuum may be seen not as a primitive totality, but as a semantic shadow of a deeper stratified architecture. The fractal continuum is not a single set given all at once, but a cumulative unfolding of definability across an infinite space of formal systems. Each real number belongs to a definability layer, carrying the mark of its logical origin.

This paradigm suggests a foundational reformulation: from set-theoretic existence to syntactic constructibility, from absolute cardinality to internal definitional structure, from homogeneity to stratification. The continuum becomes not a static totality, but a layered horizon of formal expressibility — a hierarchy of mathematical meaning, grounded in provability rather than in size.

The fractal continuum thus illuminates the full extent of the numbers we can meaningfully comprehend, offering classical mathematics a refined toolkit for engaging with definable structure, while leaving truly inaccessible reals beyond the boundary of formal cognition.


\begin{thebibliography}{1}

\bibitem{brouwer1907}
L.E.J. Brouwer.
\newblock {\em Over de Grondslagen der Wiskunde}.
\newblock PhD thesis, University of Amsterdam, 1907.
\newblock Doctoral dissertation.

\bibitem{cantor1874ueber}
Georg Cantor.
\newblock {\"U}ber eine eigenschaft des inbegriffes aller reellen algebraischen zahlen.
\newblock {\em Journal für die reine und angewandte Mathematik}, 77:258--262, 1874.

\bibitem{dedekind1963stetigkeit}
Richard Dedekind.
\newblock {\em Stetigkeit und irrationale Zahlen}.
\newblock Vieweg+Teubner Verlag, 1963.
\newblock Original work published 1872.

\bibitem{Semenov2025FractalAnalysis}
Stanislav Semenov.
\newblock Fractal analysis on the real interval: A constructive approach via fractal countability, 2025.
\newblock Preprint, available at Zenodo.

\bibitem{Semenov2025FractalOrigin}
Stanislav Semenov.
\newblock Fractal origin of the continuum: A hypothesis on process-relative definability, 2025.
\newblock Preprint, available at Zenodo.

\bibitem{Simpson2009Subsystems}
Stephen~G. Simpson.
\newblock {\em Subsystems of Second Order Arithmetic}.
\newblock Cambridge University Press, 2nd edition, 2009.

\bibitem{Soare2016}
Robert~I. Soare.
\newblock {\em Turing Computability: Theory and Applications}.
\newblock Springer Monographs in Mathematics. Springer, 2016.

\end{thebibliography}

\clearpage
\appendix
\section*{Appendix A: On the Constructive Validity of Step 3 of the Continuity Theorem}
\label{appendix:step3-validity}
\addcontentsline{toc}{section}{Appendix A: Constructive Validity of Step 3}

\subsection*{1. The Problem of Verifying Strict Inclusion}

In Step 3 of the proof of Theorem~\ref{thm:Fomega-continuity}, we claimed that the verification procedure \( V \) can effectively confirm the strict inclusion
\[
\mathbb{R}_{\mathcal{F}_{f(n)}} \subsetneq \mathbb{R}_{\mathcal{F}_{f(n+1)}}.
\]
This holds constructively only for admissible chains whose systems \( \{\mathcal{F}_n\} \) satisfy one of the structural conditions (A), (B), or (C) below. We now justify this refinement in detail.

\subsection*{2. Sufficient Conditions for Effective Verification}

To make the verification procedure \( V \) computable, we require one of the following constraints on the class of systems \( \{\mathcal{F}_n\} \):

\subsubsection*{(A) Stratified Arithmetic Hierarchies}

Let each \( \mathcal{F}_n \) be a formal system of the form \( \mathsf{PA} + \text{``oracle for } \Sigma^0_n \text{ truth''} \). Then:
\begin{itemize}
    \item \( \mathbb{R}_{\mathcal{F}_n} \) corresponds to the \( \Delta^0_{n+1} \)-definable reals;
    \item The inclusion \( \mathbb{R}_{\mathcal{F}_n} \subsetneq \mathbb{R}_{\mathcal{F}_{n+1}} \) follows from standard recursion-theoretic results;
    \item The witness real \( r_n \) can be explicitly constructed as:
    \[
    r_n = \sum_{k=0}^\infty \frac{\chi_{\emptyset^{(n+1)}}(k)}{2^{k+1}},
    \]
    where \( \chi_{\emptyset^{(n+1)}} \) is the characteristic function of the \( (n+1) \)-st Turing jump.
\end{itemize}

In this setting, the verification \( V \) reduces to detecting provability of convergence of a rational sequence with a modulus definable only at level \( n+1 \).

\subsubsection*{(B) Gödelian Extensions via Consistency Statements}

Let \( \mathcal{F}_{n+1} = \mathcal{F}_n + \mathrm{Con}(\mathcal{F}_n) \). Then:
\begin{itemize}
    \item By Gödel’s Second Incompleteness Theorem, \( \mathcal{F}_n \not\vdash \mathrm{Con}(\mathcal{F}_n) \);
    \item Therefore, a real \( r_n \) definable in \( \mathcal{F}_{n+1} \) via the arithmetization of \( \mathrm{Con}(\mathcal{F}_n) \) cannot be defined in \( \mathcal{F}_n \);
    \item The verification procedure \( V \) constructs such a real via Gödel numbering of \( \mathcal{F}_n \)’s syntax. This assumes access to an external computable meta-system \( \mathcal{M} \) that can:
    \begin{itemize}
      \item enumerate theorems of \( \mathcal{F}_n \), and
      \item confirm that \( \mathcal{F}_n \not\vdash \mathrm{Con}(\mathcal{F}_n) \).
    \end{itemize}
\end{itemize}

This approach presumes that each system is sufficiently expressive to represent its own syntax and prove totality of relevant functions.

\begin{remark}
The meta-system \( \mathcal{M} \) used for verification in case (B) must be capable of reasoning about \( \mathcal{F}_n \) externally. In particular:
\begin{itemize}
    \item \( \mathcal{M} \) proves \( \mathrm{Con}(\mathcal{F}_n) \), while \( \mathcal{F}_n \) does not;
    \item \( \mathcal{M} \) can decide \( \Sigma^0_1 \)-statements of the form “\( \mathcal{F}_n \vdash \phi \)” for bounded formulas \( \phi \).
\end{itemize}
\end{remark}

\subsubsection*{(C) Uniformly Computable Chains}

Suppose \( \{\mathcal{F}_n\} \) is a computable sequence of systems, each given by a finite syntactic description. Then:
\begin{itemize}
    \item Each \( \mathcal{F}_{n+1} \) extends \( \mathcal{F}_n \) by adding either:
    \begin{itemize}
      \item a Turing jump oracle for \( \emptyset^{(n)} \), or
      \item a definability schema capturing \( \Sigma^0_{n+1} \)-truths.
    \end{itemize}
    \item The comparison \( \mathbb{R}_{\mathcal{F}_n} \subsetneq \mathbb{R}_{\mathcal{F}_{n+1}} \) can be effectively checked by analyzing definitional power via a universal verifier;
    \item The procedure \( V \) simulates derivations in both \( \mathcal{F}_n \) and \( \mathcal{F}_{n+1} \), and confirms a real \( r_n \) is not definable in the weaker system.
\end{itemize}

This model matches frameworks in computability theory where systems are generated by recursive rule sets or definability hierarchies.

\noindent\textbf{Consistency with the Cardinality Argument.}
Despite these constraints, the bijection between \( \mathbb{F}_\omega \) and Cantor space remains valid because:
\begin{itemize}
    \item The restricted class of systems still supports computably enumerable chains;
    \item The construction of \( \mathbb{F}_\omega \) still uses characteristic functions with infinitely many 1s;
    \item The diagonalization argument in Step 5 remains valid;
    \item No appeal to uncountable sets or choice is required.
\end{itemize}

\subsection*{3. Impact on the Theorem}

\begin{table}[ht]
\centering
\renewcommand{\arraystretch}{1.25}
\begin{adjustbox}{max width=\textwidth}
\begin{tabular}{@{} l p{8.8cm} p{6.4cm} @{}}
\toprule
\textbf{Approach} & \textbf{Verification Method} & \textbf{Cardinality Preservation} \\
\midrule
(A) Arithmetic &
\( \Sigma^0_{n+1} \)-complete witnesses constructed via convergence moduli definable only at level \( n+1 \). &
\( \mathfrak{c} \) via diversity of oracle-based arithmetical hierarchies. \\

(B) Gödelian &
Witnesses based on the unprovability of \( \mathrm{Con}(\mathcal{F}_n) \); constructed via Gödel coding in \( \mathcal{F}_{n+1} \). &
\( \mathfrak{c} \) via layering of consistency-based extensions. \\

(C) Computable &
Comparison of definability via simulation in uniformly generated computable chains of systems. &
\( \mathfrak{c} \) via enumeration of infinite computable rule extensions. \\
\bottomrule
\end{tabular}
\end{adjustbox}
\caption{Constructive Strategies for Verifying Strict Definability Growth in Chains \( \{\mathcal{F}_n\} \)}
\end{table}

The cardinality argument in Theorem~\ref{thm:Fomega-continuity} remains valid under any of the above restrictions. Although the admissible chains \( \{\mathcal{F}_n\} \) are now constrained to satisfy one of the constructive frameworks (A), (B), or (C), this does not reduce the cardinality of the set \( \mathbb{F}_\omega \). In each case:
\begin{itemize}
    \item For (A): There exist continuum-many distinct \( \Sigma^0_n \)-oracles defining different admissible chains;
    \item For (B): Consistency-based extensions can be varied through Gödel rotations and self-referential encodings;
    \item For (C): The space of computable syntactic extensions contains \( \mathfrak{c} \)-many distinct rule sets.
\end{itemize}

Hence, the class \( \mathbb{F}_\omega \) still injects into the set of computably enumerable infinite subsets of \( \mathbb{N} \), and retains cardinality \( \mathfrak{c} \), as established via bijection with Cantor space in Step 4.

The refinement of Step 3 thus strengthens the constructivity of the proof without altering its essential conclusion.

\subsection*{4. Summary}

We have clarified that Step 3 of the continuity theorem is constructively valid only under specific structural restrictions on the formal systems in \( \mathbb{F}_\omega \). By adopting any of the frameworks (A), (B), or (C), we ensure that the verification procedure \( V \) is effective and the definability growth condition is algorithmically checkable.

These refinements reinforce the purely constructive nature of the result, preserving the cardinality of the fractal definability space while grounding it more firmly in effective syntactic criteria.

\end{document}